\documentclass[aos,preprint]{imsart}

\RequirePackage{amsthm,amsmath,amsfonts,amssymb}
\RequirePackage[numbers]{natbib}
\RequirePackage[colorlinks,citecolor=blue,urlcolor=blue]{hyperref}
\RequirePackage{graphicx}
\RequirePackage{xcolor}
\RequirePackage{commath}

\startlocaldefs
\newtheorem{theorem}{Theorem}[section]
\newtheorem{proposition}[theorem]{Proposition}
\newtheorem{corollary}[theorem]{Corollary}
\newtheorem{lemma}[theorem]{Lemma}
\usepackage{changes}
\theoremstyle{remark}
\newtheorem{definition}[theorem]{Definition}

\ifx\example\undefined
\newtheorem{example}[theorem]{Example}
\newtheorem{remark}[theorem]{Remark}

\fi
\newcommand{\R}{{\mathbb R}}
\newcommand{\pfc}{{\mathcal F_c}}

\DeclareMathOperator{\med}{med}
\DeclareMathOperator{\Med}{Med}

\endlocaldefs

\begin{document}
	
	\begin{frontmatter}
		\title{Statistical depth and support medians for fuzzy data}
		%\title{A sample article title with some additional note\thanksref{t1}}
		\runtitle{Fuzzy depth median}
		%\thankstext{T1}{A sample additional note to the title.}
		
		\begin{aug}
			%%%%%%%%%%%%%%%%%%%%%%%%%%%%%%%%%%%%%%%%%%%%%%
			%%Only one address is permitted per author. %%
			%%Only division, organization and e-mail is %%
			%%included in the address.                  %%
			%%Additional information can be included in %%
			%%the Acknowledgments section if necessary. %%
			%%%%%%%%%%%%%%%%%%%%%%%%%%%%%%%%%%%%%%%%%%%%%%
			\author[A]{\fnms{LUIS} \snm{GONZ\'ALEZ-DE LA FUENTE}\ead[label=e1]{gdelafuentel@unican.es}},
			\author[A]{\fnms{ALICIA} \snm{NIETO-REYES}\ead[label=e2]{alicia.nieto@unican.es}\thanks{Corresponding author, ORCID: 0000-0002-0268-3322}}
			\and
			\author[B]{\fnms{PEDRO} \snm{TER\'AN}\ead[label=e3]{teranpedro@uniovi.es}}
			%%%%%%%%%%%%%%%%%%%%%%%%%%%%%%%%%%%%%%%%%%%%%%
			%% Addresses                                %%
			%%%%%%%%%%%%%%%%%%%%%%%%%%%%%%%%%%%%%%%%%%%%%%
			\address[A]{Departamento de Matem\'aticas, Estad\'istica y Computaci\'on,
				Universidad de Cantabria, Spain
				%\printead{e1,e2}
			}
			
			\address[B]{Universidad de Oviedo, Spain
				%\printead{e3}
			}
		\end{aug}
		
		\begin{abstract}
					Statistical depth functions order the elements of a space with respect to their centrality in a probability distribution or dataset. Since many depth functions are maximized in the real line by the median, they provide a natural approach to defining median-like location estimators for more general types of data (in our case, fuzzy data). We analyze the relationships between depth-based medians, medians based on the support function, and some notions of a median for fuzzy data in the literature. We take advantage of specific depth functions for fuzzy data defined in our former papers: adaptations of Tukey depth, simplicial depth, $L^1$-depth and projection depth.
       	 \end{abstract}
		
		\if0
		...........
		..........
		\fi
		
		
	\end{frontmatter}
	%%%%%%%%%%%%%%%%%%%%%%%%%%%%%%%%%%%%%%%%%%%%%%
	%% Please use \tableofcontents for articles %%
	%% with 50 pages and more                   %%
	%%%%%%%%%%%%%%%%%%%%%%%%%%%%%%%%%%%%%%%%%%%%%%
	%\tableofcontents
	
	\section{Introduction}
	
	Statistical depth functions order the elements of a space with respect to a distribution on the space, and, specifically, the elements of a dataset. This order is trivial in dimension one, providing the concept of quantiles and,  in particular, of median.  However,  it is not trivial in higher dimensions.
	A way to see it  is that there exist multiple manners of extending the idea of median, and other quantiles, to spaces of dimension larger than one while satisfying certain desirable properties.
	These desirable properties are gathered in  axiomatic notions of statistical depth, see \cite{ZuoSerfling} for multivariate spaces,  \cite{NietoBattey} for functional, metric, spaces and \cite{primerarticulo} for spaces of fuzzy sets. Each notion takes into account the particular characteristics of the underlying space. Examples of depth functions are: Tukey depth or halfspace depth \cite{tukey}, random Tukey depth \cite{randomTukey} and simplicial depth \cite{LiuSimplicial} in multivariate spaces, spatial depth \cite{indiosespacial}, metric depth \cite{metricdepth} and lens metric depth \cite{lensmetric} in functional spaces and fuzzy Tukey depth \cite{primerarticulo}, fuzzy simplicial-type depths \cite{simplicialsegundo}, fuzzy projection depth \cite{proyecciontercero} and fuzzy $L^{1}$-type depth \cite{proyecciontercero} in fuzzy spaces.

	A naive but straightforward way of computing quantiles, in, for instance,  multivariate spaces, 
	is the component-wise quantiles, and, in particular, the component-wise median. However, this has some undesired properties: (I)  The coordinate-wise median of a dataset does not necessarily   belong to the convex hull of the dataset \cite{Cabrera}. (II) When computed with respect to the vertices of an equilateral triangle and its center of mass, the component-wise median depends on the coordinate axis and it is not necessarily the center of mass \cite{Rafa}. This undesired behavior of naive strategies has resulted in that lately the median induced by the different depth functions is used in the literature to generalized the, one-dimensional, median. Examples are \cite{medianM1}, \cite{medianM2} or \cite{medianM3} in multivariate spaces and  \cite{LopezRomoBand}, \cite{medianF1} or \cite{medianF2} in functional spaces. This paper is dedicated to propose and study, among others, the medians induced by the existing fuzzy depth functions.

 In the fuzzy setting, the literature contains  already some notions of median and quantile functions of fuzzy random variables, which are not induced in their conceptualization by stiatitical fuzzy depth functions.
 Each of them is defined for the specific case of $\mathcal{F}_{c}(\mathbb{R})$. In \cite{GregorMedian}, Grzegorzewski proposes the definition of a median of a fuzzy random variable based on the Zadeh's extension principle. More recently, Sinova et al. \cite{medianfuzzy1} propose the notion of median of a fuzzy random variable (the 1-median) from an $L^1$-minimization principle.
 Also Shvedov in \cite{ShvedovMediana} proposes the definition of a quantile function of a fuzzy random variable, which assigns to every $p\in (0,1)$ a fuzzy set. In that setting, if we take $p = 1/2$, we obtain the median of the fuzzy random variable. 
 In the present paper we focus our study to the concepts of medians proposed by \citet{medianfuzzy1} as a minimization problem and by Grzegorzewski \cite{GregorMedian} using the Zadeh's extension principle and connect them with our %
 proposals. 
 This is due to these proposal  are more consistent with the fuzzy theory than the median in \cite{ShvedovMediana}, which is a generalization of the univariate case.

	The paper is organized as follows. Section \ref{preliminaries} contains the notation and background, on fuzzy sets and depth, necessary for a comprehensive understanding of the paper. In Section \ref{support} we present two definitions of fuzzy medians, one based on   support functions and the other  on maximizers of depth functions. 
	A study about the relation between them and with  the existing notions is pursued in Sections \ref{support} and \ref{max}. All proofs and further results are contained in Section \ref{proofs}. A discussion is in Section \ref{discussion}.

	\section{Preliminaries}\label{preliminaries}
	
	The following notation is used throughout.
	We denote by $\mathcal{K}_{c}(\mathbb{R}^{p})$ the set of all non-empty, compact and convex subsets of $\mathbb{R}^{p}$. A fuzzy set on $\mathbb{R}^p$ is any function  $A: \mathbb{R}^{p}\rightarrow [0,1]$. The $\alpha$-levels, or $\alpha$-cuts, of a fuzzy set $A$ are defined as the sets $A_\alpha = \{x\in\mathbb{R}^{p}: A(x)\geq\alpha\}$ with $\alpha\in (0,1]$ and $A_0 = \text{clo}(\{x\in\mathbb{R}^{p}: A(x)>0\})$, where $\text{clo}(\cdot)$ denotes the closure of a set. The set of fuzzy sets whose $\alpha$-levels are elements of $\mathcal{K}_{c}(\mathbb{R}^{p})$ is denoted by $\mathcal{F}_{c}(\mathbb{R}^{p})$. Abusing of nomenclature, we refer to the elements of $\mathcal{F}_{c}(\mathbb{R}^{p})$ as fuzzy sets.

	Ordinary sets, $K\in\mathcal{K}_{c}(\mathbb{R}^{p})$, can be identified with its \textit{indicator function}, $\text{I}_{K} : \mathbb{R}^{p}\rightarrow\mathbb{R}$, where $\text{I}_{K}(x) = 1$ if $x\in K$ and $\text{I}_{K}(x) = 0$ otherwise. The sets of the form $\text{I}_{K}$ are called \textit{crisps sets}. Thus, abusing again of the notation, we consider that $\mathcal{K}_{c}(\mathbb{R}^{p})\subseteq\mathcal{F}_{c}(\mathbb{R}^{p})$.
	
	Let us denote by $\mathbb{S}^{p-1} := \{x\in\mathbb{R}^{p}: \|x\|\leq 1 \}$ the unit sphere on $\mathbb{R}^{p}$, with $\|.\|$ the euclidean norm on $\mathbb{R}^{p}$ and by $\langle \cdot ,\cdot\rangle$ the usual inner product on $\mathbb{R}^{p}$. The mapping $s_{A}: \mathbb{S}^{p-1}\times [0,1]\rightarrow\mathbb{R}$ defined as
	$
	s_{A}(u,\alpha) = \sup\{\langle u,v\rangle : v\in A_{\alpha}\}
	$
	for every $u\in\mathbb{S}^{p-1}$ and $\alpha\in [0,1]$ is known as the \textit{support function} of the fuzzy set $A\in\mathcal{F}_{c}(\mathbb{R}^{p})$.

	It is common to consider \textit{triangular fuzzy sets} in the case of $\mathcal{F}_{c}(\mathbb{R})$ (see \cite[Section 4.1]{Klir}), also known as triangular fuzzy numbers. %
	Let $a, b, c \in\mathbb{R}$ with  $a\leq b\leq c,$ the triangular fuzzy number $T(a,b,c)$ is given by
	
	\begin{equation*}\label{triangularfuzzyset}
		T(a,b,c)(x) := \left\{ \begin{array}{lcr}
			\cfrac{x - a}{b-a}  & \text{ if } x\in [a,b] \text{ and } a < b, %\\
			\\ \cfrac{x - c}{b-c}  &  \text{ if } x\in [b,c] \text{ and } b < c %\\
			\\ 1 & \text{ if } x\in [a,b] \text{ and } a = b \text{ or } x\in [b,c] \text{ and } b = x, %\\
			\\ 0 & \text{otherwise.}
		\end{array}
		\right.
	\end{equation*}

\subsection{Zadeh's extension principle}\label{Aarith}
It is possible to define new fuzzy sets from non-fuzzy/crisp functions acting over fuzzy sets, using the Zadeh's extension principle \cite{zadehextension}. Given $A\in\mathcal{F}_{c}(\mathbb{R}^{p})$ and a function $f : \mathbb{R}^{p}\rightarrow\mathbb{R}^{p},$  the image $f(A)$ is  defined to be
$f(A)(t) = \sup\{A(y) : y\in\mathbb{R}^{p}, f(y) = t \}$
for all $t\in\mathbb{R}^{p}$.
If we consider an arbitrary $f$, it could happen that the set $f(A)$ does not belong to $\mathcal{F}_{c}(\mathbb{R}^{p})$. Thus we only consider continuous functions.

\subsection{Metrics in the fuzzy setting}
\label{MFS}
It is possible to endow $\mathcal{F}_{c}(\mathbb{R}^{p})$ of a metric space structure, using for instance
the well-known % 
familiy of $\rho_{r}$ metrics, which are $L^{r}$-type metrics (see \cite{diamondkloden}). The $\rho_{r}$ metric for any $A,B\in\mathcal{F}_{c}(\mathbb{R}^{p})$ is defined as
\begin{equation}\label{rhor}
	\rho_{r}(A,B) := \left(\int_{\mathbb{S}^{p-1}}\int_{[0,1]}\|s_{A}(u,\alpha)-s_{B}(u,\alpha)\|^{r}d\mathcal{V}^{p}(u)\, d\nu(\alpha)\right)^{1/r},
\end{equation}
where $\nu$ represents the Lebesgue measure and $\mathcal{V}_{p}$ denotes the normalized Haar measure in $\mathbb{S}^{p-1}$. %

\subsection{Fuzzy random variables}\label{Frv}

Let $(\Omega,\mathcal{A},\mathbb{P})$ be a probability space, a random compact set is defined as a function $\Gamma:\Omega\rightarrow\mathcal{K}_{c}(\mathbb{R}^{p})$ such that the set $\{\omega\in\Omega : \Gamma(\omega)\cap K\neq\emptyset \}$ is an element of $\mathcal{A}$, whenever $K\in\mathcal{K}_{c}(\mathbb{R}^{p})$ (see \citet{Mol}). Based on \citet{PuriRalescu} a fuzzy random variable is defined as a function $\mathcal{X}:\Omega\rightarrow\mathcal{F}_{c}(\mathbb{R}^{p})$ such that $\mathcal{X}_{\alpha}(\omega)$ is a random compact set for all $\alpha\in[0,1]$, where $\mathcal{X}_{\alpha}:\Omega\rightarrow\mathcal{P}(\mathbb{R}^{p})$ is defined as the $\alpha$-cut of every image of $\mathcal{X}$, i.e. $\mathcal{X}_{\alpha}(\omega) := \{x\in\mathbb{R}^{p}: \mathcal{X}(\omega)(x)\geq\alpha  \}$ for any $\omega\in\Omega$. In the case of $p = 1$, we refer to the fuzzy random variables as fuzzy random numbers.

The class of fuzzy random variables $\mathcal{X}$  on the measurable space $(\Omega,\mathcal{A})$ such that $\text{E}[\|\mathcal{X}_{0}\|^{r}] < \infty$ whenever $r\in [1,\infty)$ is denoted  by $L^{r}[\mathcal{F}_{c}(\mathbb{R})^{p}]$. Thus, $L^{0}[\mathcal{F}_{c}(\mathbb{R})^{p}]$ is the more general class of fuzzy random variables. 

Given    $\mathcal{X}\in L^{0}[\mathcal{F}_{c}(\mathbb{R})^{p}],$ its support function $s_{\mathcal{X}} : \mathbb{S}^{p-1}\times [0,1]\times\Omega\rightarrow\mathbb{R}$ is defined by
$ 
	s_{\mathcal{X}}(u,\alpha,\omega) := s_{\mathcal{X}(\omega)}(u,\alpha).
$ 

The function $s_{\mathcal{X}}(u,\alpha):\Omega\rightarrow\mathbb{R}$ is a real random variable whenever $u\in\mathbb{S}^{p-1}$ and $\alpha\in [0,1]$ \cite{jointPedroMiriam}. We denote by $F_{u,\alpha}$ the distribution function of $s_{\mathcal{X}}(u,\alpha)$.

\subsection{Depth functions for fuzzy sets}\label{instances}
 
 Let $\mathcal{H}\subseteq L^{0}[\mathcal{F}_{c}(\mathbb{R}^{p})]$ and $\mathcal{J}\subseteq\mathcal{F}_{c}(\mathbb{R}^{p}).$ The following depths are functions 
 of the type $D:\mathcal{J}\times\mathcal{H}\rightarrow [0,\infty)$ evaluated at a fuzzy set $A\in\mathcal{J}$ with respect to a fuzzy random variable $\mathcal{X}\in\mathcal{H}.$
For any $r\in [1,\infty)$,  the \textit{r-natural depth}  \cite{proyecciontercero}  is 
\begin{equation}\label{D_r}
D_{r}(A;\mathcal{X}) = (1 + \text{E}[\rho_{r}(A;\mathcal{X})])^{-1}
\end{equation}	
with $D_{r}(A;\mathcal{X}) = 0$ in case of expectation being infinite.
 The \textit{Tukey depth} \cite{primerarticulo} is defined as
\begin{equation}\label{T}
D_{FT}(A;\mathcal{X}) = \inf_{u\in\mathbb{S}^{p-1}, \alpha\in [0,1]} \min(\mathbb{P}[\omega\in\Omega: \mathcal{X}(\omega)\in S_{u,\alpha}^{-}],\mathbb{P}[\omega\in\Omega: \mathcal{X}(\omega)\in S_{u,\alpha}^{+}]),
\end{equation}
where
		$S_{u,\alpha}^{-} %&
		= \{U\in\mathcal{F}_{c}(\mathbb{R}^{p}) : s_{U}(u,\alpha) - s_{A}(u,\alpha)\leq 0\}$ and %\\ \nonumber
		$S_{u,\alpha}^{+} %&
		= \{U\in\mathcal{F}_{c}(\mathbb{R}^{p}) : s_{U}(u,\alpha) - s_{A}(u,\alpha)\geq 0\}.$
	The \textit{projection depth function}  \cite{proyecciontercero} is defined as $D_{FP}(A;\mathcal{X}) = (1 + O(A;\mathcal{X}))^{-1}$, where
\begin{equation}\label{O}
O(A;\mathcal{X}) = \sup_{(u,\alpha)\in\mathbb{S}^{p-1}\times [0,1]}\cfrac{|s_{A}(u,\alpha) - \med(s_{\mathcal{X}}(u,\alpha))|}{\text{MAD}(s_{\mathcal{X}}(u,\alpha))}.
\end{equation}	
 $\med(\cdot)$ denotes the median of a real random variable, with the usual convention of taking the mid-point of the interval of medians when the median is not unique, and $\text{MAD}(\cdot)$ the corresponding median absolute deviation. 
 
Throughout the paper, for both crisp and fuzzy deta we will follow this criterion: medians which are defined so as to be unique will be represented by $\med$ while sets of medians which may potentially include more than one point will be represented by $\Med$. Thus $\Med(X)$ represents the interval $[\underline{\med}(X),\overline{\med}(X)]$ formed by all medians of a random variable $X$. If the median of $X$ is unique, then $\Med(X)=\{\med(X)\}$. This convention will be very useful since several medians for fuzzy data will be considered.

	The \textit{modified simplicial depth}  \cite{simplicialsegundo}  is defined as
\begin{equation}\label{ms}
D_{mS}(A;\mathcal{X}) = \text{E}(\mathcal{V}_{p}\otimes\nu\{(u,\alpha)\in\mathbb{S}^{p-1}\times[0,1] : s_{A}(u,\alpha)\in [m(u,\alpha),M(u,\alpha)]\}),
\end{equation}
where $\mathcal{X}_{1},\ldots ,\mathcal{X}_{p+1}$ are independent and identically distributed random variables distributed as $\mathcal{X}$, $m(u,\alpha) = \min\{\mathcal{X}_{1},\ldots ,\mathcal{X}_{p+1}\}$ and $M(u,\alpha) = \max\{\mathcal{X}_{1},\ldots ,\mathcal{X}_{p+1}\}$.
The \textit{fuzzy simplicial depth}  \cite{simplicialsegundo} is 
\begin{equation}\label{fs}
D_{FS}(A;\mathcal{X}) = \inf_{u\in\mathbb{S}^{p-1}}\text{E}(\nu\{(u,\alpha)\in\mathbb{S}^{p-1}\times[0,1] : s_{A}(u,\alpha)\in [m(u,\alpha),M(u,\alpha)]\}).
\end{equation}

\section{Definition of support medians and depth medians}\label{support}\label{def}

It has long been recognized that a sample of ordered linguistic values modeled by fuzzy sets admits a median (see, e.g., \cite{NguWu}). However, that is just the median of the linguistic values themselves when regarded as an ordinal variable. It is interesting to have ways to compute a median when the fuzzy values in the sample are not obviously ordered or do not reflect an order known in advance; specially, given the ever-increasing list of methods for ranking fuzzy numbers (see, e.g., \cite{BorDeg,Den,Ngu}).

In this section, we introduce two proposals to find plausible candidates for the median of a fuzzy random variable: support medians and depth medians. Later in the section, we present depth medians as the maximizers of a statistical depth function. 
Support medians include the medians defined by \citet{GregorMedian} and \citet{medianfuzzy1}. Thus, we provide these definitions, before introducing the support median notion.

\citet{krusemeyer} introduced the idea of a \textit{fuzzy parameter} of a fuzzy random variable. In that setting, \citet{GregorMedian} defined the median as the \textit{perception of the median} of a fuzzy random variable. Formally, a {\em Grzegorzewski median} of $\mathcal{X}\in L^{0}[\mathcal{F}_{c}(\mathbb{R})]$ is a fuzzy number, $\med_{Gr}(\mathcal{X})$, such that for every $t\in\mathbb{R}$, its membership functions fulfills
	\begin{equation}\label{medianaGregor} 
	\med_{Gr}(\mathcal{X})(t) = \sup\{\inf_{\omega\in\Omega}\mathcal{X}(\omega)(X(\omega)) : X\in\mathcal{B}_{\mathcal{X}},  t\in\med(X)\}.
	\end{equation}
	There, $\mathcal{B}_{\mathcal{X}}$  is  the set of Borel-measurable mappings $X : (\Omega,\mathcal{A})\rightarrow (\mathbb{R},\beta(\mathbb{R}))$, where  $(\Omega,\mathcal{A},\mathbb{P})$ is the probabilistic space associated with the fuzzy random variable, $\mathcal{X} : \Omega\rightarrow\mathcal{F}_{c}(\mathbb{R}),$  and $\beta(\mathbb{R})$ is the Borel $\sigma$-algebra on $\mathbb{R}.$

\citet{GregorMedian} proved a useful characterization: The fuzzy median $\med_{Gr}(\mathcal{X})$ is uniquely given by
	\begin{equation}\label{propGregor} 
	(\med_{Gr}(\mathcal{X}))_{\alpha} = [\underline{\med}(\inf\mathcal{X}_{\alpha}),\overline{\med}(\sup\mathcal{X}_{\alpha})]
	\end{equation}
	for every $\alpha\in [0,1]$.

As Grzegorzewski's median  is defined using Zadeh's extension principle, it is unique. However, the one dimensional median is not necesarily unique. Additionally, if we define the  median as the fuzzy number which maximizes a certain function measuring central tendency, we may obtain non-unique  medians. Following this idea, a definition for the median of a fuzzy random number is proposed by \citet{medianfuzzy1}.

The  {\em 1-median} of $\mathcal{X}\in L^{1}[\mathcal{F}_{c}(\mathbb{R})]$ is the set of minimizers
\begin{eqnarray}\label{1medianfuzzynumber} 
\Med_1(\mathcal{X}) = \arg\min_{U\in\mathcal{F}_{c}(\mathbb{R})}\text{E}[\rho_{1}(U,\mathcal{X})].
\end{eqnarray}
Each element of $\Med_1(\mathcal X)$ is also called a {\em $1$-median}. \citet{medianfuzzy1} show that  in $\pfc(\R)$ there always exist $1$-medians, by constructing a fuzzy set $\med_{Si}(\mathcal X)$ with the property
$$s_{\med_{Si}(\mathcal X)}(u,\alpha) = \med(s_{\mathcal X}(u,\alpha))\quad\forall u\in 
\mathbb{S}^{0},\alpha\in [0,1].$$
Since an element of $\pfc(\R)$ is identified by its support function, $\med_{Si}(\mathcal X)$ is unique in having that property.
Moreover, \cite[Theorem 4.1.]{medianfuzzy1} states that
	\begin{eqnarray}\label{teoremaBea}
	\med_{Si}(\mathcal X)\in \Med_1(\mathcal X).
	\end{eqnarray}

It is hard to say that $\med_{Gr}$   is better or worse than $\med_{Si}$. In some cases, the Grzegorzewski median does not behave like the ordinary (midpoint) median in $\R$. For example, the median of $\text{I}_{\{0\}}$ and $\text{I}_{\{2\}}$, which represent the crisp points $0$ and $2$, is not $\text{I}_{\{1\}}$ but $\text{I}_{[0,2]}$. That is a consequence of the definition via Zadeh's extension principle, whence $\med_{Gr}(\mathcal X)(x)$ represents the degree to which $x$ is a median of $\mathcal X$. In the example, every point between 0 and 2 can fully be considered a median and from this point of view the result is sensible. On the other hand, in this example $\med_{Si}$ is $\text{I}_{\{1\}}$, which replicates the behaviour of the (midpoint) median in $\R$ as a unique representative. It can also be said that the information about other medians captured by $\med_{Gr}$ is lost.

Let us introduce now the notion of  {\em support median}. 
\begin{definition}
A fuzzy set $A\in\pfc(\R^p)$ is a {\em support median} of a fuzzy random variable $\mathcal X\in L^{0}[\mathcal{F}_{c}(\mathbb{R}^p)]$ if  $s_A(u,\alpha)$ is a median of $s_{\mathcal X}(u,\alpha)$ for all $u\in \mathbb{S}^{p-1}$ and $\alpha\in [0,1].$  That is, if $s_{A}(u,\alpha)\in\med(s_{\mathcal{X}}(u,\alpha))$ for all $u\in \mathbb{S}^{p-1}$ and $\alpha\in [0,1]$. 
The set of the support medians of $\mathcal X$ is denoted $\Med_s(\mathcal X)$.
\end{definition}
With the results of \citet{medianfuzzy1}, it is easily obtained for $\mathcal{X}\in L^{1}[\mathcal{F}_{c}(\mathbb{R})]$  that 
\begin{eqnarray}\label{e1}
\med_{Si}(\mathcal X)\in\Med_s(\mathcal X).
\end{eqnarray}
 For $\mathcal{X}\in L^{0}[\mathcal{F}_{c}(\mathbb{R})]$ it is also  easy to check 
\begin{eqnarray}\label{e2}
\med_{Gr}(\mathcal X)\in\Med_s(\mathcal X).
\end{eqnarray}
  Moreover, one has the following result where we have that each $\alpha$-level of a support median is contained in the corresponding $\alpha$-level of the Grzegorzewski fuzzy median.
\begin{proposition}\label{gregorsupport}
	Let $\mathcal{X}\in L^{0}[\mathcal{F}_{c}(\mathbb{R})].$ %
Then, $A_\alpha\subseteq(\text{med}_{Gr})_\alpha$ for every $A\in\Med_s(\mathcal X)$ and $\alpha\in [0,1].$
\end{proposition}

Thus $\med_{Gr}(\mathcal X)$ has the maximality property that each support median of $\mathcal X$ must be contained in it.

The results in \eqref{e1}, \eqref{e2} and Proposition \ref{gregorsupport}   hold only for the one-dimensional ($p=1$) case. For $p>1$, there exist fuzzy random variables without any support median, as we show in the following example.

\begin{example}
	Given the probability space $(\{\omega_{1},\omega_{2},\omega_{3}\},\mathcal{P}(\{\omega_{1},\omega_{2},\omega_{3}\}),\mathbb{P})$, with $\mathbb{P}(\{\omega_{1}\}) = \mathbb{P}(\{\omega_{2}\}) = \mathbb{P}(\{\omega_{3}\})$, let $\mathcal{X} : \{\omega_{1},\omega_{2},\omega_{3}\}\rightarrow\mathcal{F}_{c}(\mathbb{R}^{2})$ be a fuzzy random variable defined by $\mathcal{X}(\omega_{1}) = \text{I}_{\{(0,-1)\}}$, $\mathcal{X}(\omega_{2}) = \text{I}_{\{(2,0)\}}$ and $\mathcal{X}(\omega_{3}) = \text{I}_{\{(1,1)\}}$.
	
	This example presents that the fuzzy random variable $\mathcal{X}$ has no support medians, that is, $\Med_s(\mathcal X) = \emptyset$.
	Assuming for a contradiction that $\Med_s(\mathcal X)\neq\emptyset$, let $A$ be a support median of $\mathcal{X},$ i.e. 
	\begin{equation}\label{sb}
	s_{A}(u,\alpha)\in\med(s_{\mathcal{X}}(u,\alpha))\mbox{ for all }u\in \mathbb{S}^{p-1}\mbox{ and }\alpha\in [0,1].
	\end{equation} 
	
	 First of all, we prove that $A$ is the indicator function of a point, that is, $A = \text{I}_{\{(x,y)\}}$ for an $(x,y)\in\mathbb R^2$.
	 By the definition of support function, we have that
	$$
	 \med(s_{\mathcal{X}}(u,\alpha)) = \med(\{\langle u,(0,-1)\rangle, \langle u,(2,0)\rangle ,\langle u,(1,1)\rangle\})
	$$
	for every $u\in\mathbb{S}^{1}$ and $\alpha\in [0,1]$. As the median of three distinct real points is a single point, taking \eqref{sb} into account, we have that 
	$s_{A}(u,\alpha)= \med(\{\langle u,(0,-1)\rangle, \langle u,(2,0)\rangle ,\langle u,(1,1)\rangle\})$ for every $u\in\mathbb{S}^{1}$ and $\alpha\in [0,1].$ Thus, $s_{A}(u,\alpha)$ does not change with $\alpha.$ 
	Additionally, $\med(s_{\mathcal{X}}(-u,\alpha)) = -\med(s_{\mathcal{X}}(u,\alpha))$ for every $u\in\mathbb{S}^{1}$ and $\alpha\in [0,1]$. We prove that the supremum and the infimum of the projection of every $\alpha$-level of $A$ in every direction $u$ coincides, which implies that $A$ is the indicator function of a single point of $\mathbb R^2$.
	On the one hand
	$$
	\med(s_{\mathcal{X}}(u,\alpha)) = s_{A}(u,\alpha) = \sup_{v\in A_{\alpha}}\langle u,v\rangle.
	$$
	On the other hand
	\begin{equation}
		\begin{aligned}\nonumber
			&- \med(s_{\mathcal{X}}(u,\alpha)) = \med(s_{\mathcal{X}}(-u,\alpha)) = s_{A}(-u,\alpha) =\\
			&\sup_{v\in A_{\alpha}} \langle -u, v\rangle = \sup_{v\in A_{\alpha}} - \langle u, v\rangle  = - \inf_{v\in A_{\alpha}} \langle u, v\rangle.
		\end{aligned}
	\end{equation}
	Thus
	$$
	\inf_{v\in A_{\alpha}} \langle u, v\rangle = \med(s_{\mathcal{X}}(u,\alpha)) = \sup_{v\in A_{\alpha}}\langle u,v\rangle
	$$
	for every $u\in\mathbb{S}^{1}$ and $\alpha\in [0,1]$. It implies that $A = \text{I}_{\{(x,y)\}}$ for some $(x,y)\in\mathbb{R}^{2}$.
	
	Now, we prove that there is no such $(x,y)\in\mathbb R^2$ such that $A = \text{I}_{\{(x,y)\}}$. Let $(1,0)\in\mathbb{S}^{1}$. It is clear that
	$$
	x = s_{A}((1,0),\alpha) = \med(\{0,2,1\}) = 1
	$$
	for every $\alpha\in [0,1]$, thus $x = 1$. If we take $(0,1)\in\mathbb{S}^{1}$,
	$$
	y = s_{A}((0,1),\alpha) = \med(\{-1,0,1\}) = 0
	$$
	for every $\alpha\in [0,1]$ and $y = 1$.
	
	The support median $A$ must be the fuzzy number $\text{I}_{\{(1,0)\}}$. But, if we consider the direction $u = (1/\sqrt{5}, -2/\sqrt{5})\in\mathbb{S}^{1}$, 
	$$
	s_{A}(u,\alpha) = \cfrac{1}{\sqrt{5}}\neq\cfrac{2}{\sqrt{5}} = \med(\{\cfrac{2}{\sqrt{5}},\cfrac{2}{\sqrt{5}}, -\cfrac{1}{\sqrt{5}}\})
	$$
	and $s_{A}(u,\alpha)\neq\med(s_{\mathcal{X}}(u,\alpha))$ for every $\alpha\in [0,1]$. This implies that the  support median for the fuzzy random variable $\mathcal{X}$ does not exist.
\end{example}

Since the example uses indicator functions of crisp points of $\R^2$, it indicates that already in the case of (crisp) bivariate data it is not possible, in general, to define a median by requiring its projection in every direction to be a median of the corresponding projection of the random vector. Therefore this behaviour is not a `flaw' of the notion of a support median for fuzzy data.

The notion of a support median is {\em ad hoc} in the sense that its definition is natural (given the identification between an element of $\pfc(\R^p)$ and its support function) but that does not prove, by itself, that support medians will inherit the nice theoretical properties usually ascribed to the median. In the remainder of the paper, we explore the idea that the framework of statistical depth functions for fuzzy data and fuzzy random variables can adequately supply mathematical criteria by which some or all support medians are the unique fuzzy sets satisfying certain nice properties. To that end, we introduce now the definition of a {\em depth median}.

\begin{definition}
Let  $\mathcal X\in L^{0}[\mathcal{F}_{c}(\mathbb{R}^{p})]$ and $D$ be a depth function. %
We define the  {\em  set of  depth medians} (or {\em $D$-medians}),  of $\mathcal X$ based on $\mathcal{J}\subseteq\mathcal{F}_{c}(\mathbb{R}^{p})$ as
\begin{equation}\label{J}
\Med_\mathcal{J}(\mathcal X;D)=\{A\in\mathcal{J}: D(A;\mathcal X)=\max_{U\in\mathcal{J}}D(U;\mathcal X)\}, %
\end{equation}
where $D(A;\mathcal X)$ is  the depth function evaluated at $A$  with respect to the distribution of  $\mathcal X.$
The elements of $\Med_\mathcal{J}(\mathcal X;D)$ are called {\em depth medians, $D$-medians,} of $\mathcal X$ based on $\mathcal{J}$. 

When $\mathcal{J}=\pfc(\R^p),$ as occurs in the rest of this work, we drop the $\mathcal{J}$ in \eqref{J}, making use of the notation $\Med(\mathcal X;D)$ and refer to it as the set of  depth medians, $D$-medians,  of $\mathcal X.$ % 
\end{definition}

It should be understood that maximizing some depth function does not automatically imply having the usual properties of a median, such as robustness. This is exemplified in the multivariate case by zonoid depth \cite{zonoid} or the depth based on large deviations \cite{PedroDepth}. In both cases, when being in dimension one, the maximum is  attained at the mean, which is not a robust location estimator.
However it is reasonable to think that maximizing a depth function is a nice property for a location estimator, specially in view of the properties that generally constitute the notion of depth  (maximization at the center of symmetry, decreasing depth as one moves away from the center, vanishing depth when going to infinity, ...).

We would like to close this section by recalling that not all available definitions of a median for fuzzy data are support medians. For instance, \cite[Definition 4.1]{Bea2} uses a mid-ldev-rdev representation. It is based on minimizing the $L^1$-distance for an appropriate metric and thus it is a depth median for the corresponding $L^1$-depth associated to that metric.

\section{Main results} %
\label{max}

The results in this section connect the concepts in Section \ref{def}. In particular: $D$-medians ($\Med(\cdot;D)$), support medians ($\Med_s(\cdot)$), and $1$-medians ($\Med_1(\cdot)$). We start by noting that $1$-medians can easily be recast as the maximizers of the $L^1$-depth,  in \eqref{D_r}, that makes use of the $\rho_1$-metric.
\begin{proposition}\label{D1support}
	Let $\mathcal X\in L^1[\pfc(\R)].$ Then, %
	$\Med_1(\mathcal X)=\Med(\mathcal X; D_1).$
\end{proposition}
The proof easily follows from the fact that  minimizing $E[\rho_1(\cdot,\mathcal X)]$ is equivalent to maximizing $D_1(\cdot;\mathcal X)=(1+E[\rho_1(\cdot,\mathcal X)])^{-1}$. 

We  establish next the first important result in this work:  $1$-medians, or $D_1$-medians, are exactly the support medians.  This result solves the minimization problem in  \eqref{1medianfuzzynumber}.
\begin{theorem}\label{Teorema1Mediana}
	Let $\mathcal X\in L^1[\pfc(\R)].$  Then, $\Med_1(\mathcal X) = \Med_s(\mathcal X)$.
\end{theorem}

The interest of Theorem \ref{Teorema1Mediana} goes both ways. From the perspective of $1$-medians, it provides a useful characterization by which they can be obtained without the burden of solving an stochastic optimization problem over the whole $\pfc(\R)$. From the perspective of support medians, it shows that they are the unique fuzzy sets with the $L^1$-minimization property that has been used as a mathematical definition of the median in a metric space since Fr\'echet \cite[{\it position \'equiprobable}, p.227]{Fre}.

An application  of Theorem \ref{Teorema1Mediana} is in proving that $D_1$ satisfies  one of the properties constituting the notion of geometric fuzzy depth \cite{primerarticulo}. For a depth function $D,$ based on $\mathcal{H}\subseteq L^{0}[\mathcal{F}_{c}(\mathbb{R}^{p})]$ and $\mathcal{J}\subseteq\mathcal{F}_{c}(\mathbb{R}^{p}),$
this property, known as P3b, states that for each $\mathcal{X}\in\mathcal{H}$ and each $A\in\mathcal{J}$  such that  $D(A;\mathcal{X}) = \sup \{D(B;\mathcal{X}) : B\in\mathcal{J}\},$ $D$ fulfills 
\begin{eqnarray*}
	D(A;\mathcal{X})\geq D(U;\mathcal{X})\geq D(V;\mathcal{X}) 
	\end{eqnarray*}
$\mbox{for all } U,V\in\mathcal{J} \mbox{ satisfying } d(A,V) = d(A,U) + d(U,V),$	 where $d:\mathcal{F}_{c}(\mathbb{R}^{p})\times\mathcal{F}_{c}(\mathbb{R}^{p})\rightarrow[0,\infty)$ is a metric.
For the fulfillment of other properties, see  \citet{proyecciontercero}.

\begin{theorem}\label{D1rho1}
	Let $\mathcal X\in L^{1}[\mathcal{F}_{c}(\mathbb{R})].$ Then,  $D_{1}$ satisfies P3b for the $\rho_{1}$ metric.
\end{theorem}
The proof is based on that Theorem \ref{Teorema1Mediana} characterizes the fuzzy sets that maximize $D_1,$ as they are support medians. In \cite{proyecciontercero}, property P3b was studied for the $\rho_r$ metrics but results for $\rho_1$ could not be obtained with that method of proof since the $L^1$-norm between support functions is not strictly convex.

We turn now to showing that the support medians are exactly the depth medians induced by the notion of Tukey depth, in \eqref{T}. Interestingly, this characterization of support medians does not require the random variable to be in $L^{1}[\mathcal{F}_{c}(\mathbb{R})],$ %
and so it is valid for every fuzzy random number.
\begin{theorem}\label{teoremaTukey1mediana}
	Let $\mathcal X\in L^{0}[\mathcal{F}_{c}(\mathbb{R})].$ Then, $\Med_s(\mathcal X) = \Med(\mathcal X;D_{FT})$.
\end{theorem}

A direct implication of Proposition \ref{D1support} and Theorems \ref{Teorema1Mediana} and \ref{teoremaTukey1mediana}, is the following corollary.
\begin{corollary}
Let $\mathcal X\in L^1[\pfc(\R)].$ Then, 
	$\Med(\mathcal X; D_1)=\Med(\mathcal X;D_{FT}).$
\end{corollary}

\if0
...................................	
\begin{remark}
	This result shows two key points about the study of 1-median of fuzzy random variables.
	For any $\mathcal{X}$, the set $\arg\min_{U\in\mathcal{F}_{c}(\mathbb{R})}\text{E}[\rho_{1}(U,\mathcal{X})]$ is not empty and any fuzzy set $A\in\mathcal{F}_{c}(\mathbb{R})$ which satisfies that $s_{A}(u,\alpha) = Med(s_{\mathcal{X}}(u,\alpha))$ for all $u\in\mathbb{S}^{0},\alpha\in [0,1]$ is a 1-median of the distribution of $\mathcal{X}$. It is clear, from the fact that the median, $m\in\mathbb{R}$, of a real random variable, $X$, is such that $\text{E}[|X - x|] = \min_{y\in\mathbb{R}}\text{E}[|X - y|]$.
\end{remark}
...................................
\fi

We turn now to $D_{FP}$-medians, which are based on the outlier function $O(\cdot; \mathcal X)$ in \eqref{O}. In order for that function to be well defined, it is required for $\mathcal X\in L^{0}[\mathcal{F}_{c}(\mathbb{R}^p)]$ to be  non-degenerate (i.e., it is not identical almost surely to some constant $A\in \pfc(\R^p)$.)
Next result states that $D_{FP}$-medians are support medians. In fact, the unique $D_{FP}$-median is $\med_{Si}(\mathcal{X})$.

\begin{proposition}
	Let $\mathcal{X}\in L^{0}[\mathcal{F}_{c}(\mathbb{R})]$ be  non-degenerate. Then, 
	\begin{eqnarray}\label{con}
	\Med(\mathcal X;D_{FP}) = \{\med_{Si}(\mathcal{X})\}\subseteq\Med_s(\mathcal X).
	\end{eqnarray}
\end{proposition}

Next result (Theorem \ref{teoremaProjection1mediana}) studies when the inclusion in  \eqref{con} becomes an equality.  For that, the theorem has an assumption of median uniqueness for $s_{\mathcal{X}}(u,\alpha),$ which cannot be removed as the following example shows.

\begin{example}
	Given the probability space $(\{\omega_{1},\omega_{2}\},\mathcal{P}(\{\omega_{1},\omega_{2}\}),\mathbb{P})$, with $\mathbb{P}(\omega_{1}) = \mathbb{P}(\omega_{2})$, let $\mathcal{X}:\{\omega_{1},\omega_{2}\}\rightarrow\mathcal{F}_{c}(\mathbb{R})$ be a fuzzy random variable such that $\mathcal{X}(\omega_{1}) = \text{I}_{\{1\}}$ and $\mathcal{X}(\omega_{2}) = \text{I}_{\{3\}}$.
	
	We have $\med(s_{\mathcal{X}}(1,\alpha)) = [1,3]$ and $\med(s_{\mathcal{X}}(-1,\alpha)) = [-3,-1]$ for every $\alpha\in [0,1]$. Using Theorem \ref{Teorema1Mediana}, the set of 1-medians of $\mathcal{X}$ is
	$$
	\Med_1(\mathcal X) = \{U\in\mathcal{F}_{c}(\mathbb{R}) : s_{U}(u,\alpha)\in [\min\{u,3u\},\max\{u,3u\}]\text{ for all } (u,\alpha)\in\mathbb{S}^{0}\times [0,1]\}
	$$
	
	Let us consider the set $A = \text{I}_{\{2\}}\in\Med_1(\mathcal X)$. By convention about the univariate median in the definition of $D_{FP}$, we take $\med(s_{\mathcal{X}}(u,\alpha)) = 2u$ and $\text{MAD}(s_{\mathcal{X}}(u,\alpha)) = 1$ for all $(u,\alpha)\in\mathbb{S}^{0}\times [0,1]$, thus
	$$
	O(A;\mathcal{X}) = \sup_{(u,\alpha)\in\mathbb{S}^{0}\times [0,1]} |s_{A}(u,\alpha) - 2u| = 0.
	$$
	Thus $D_{FP}(A;\mathcal{X}) = 1$ and $A$ necessarily maximizes the projection depth. Now let us consider $B = T(1,2,3)\in\Med_1(\mathcal X)$. We have $B_{\alpha} = [1+\alpha,3-\alpha]$, $s_{B}(1,\alpha) = 3-\alpha$ and $s_{B}(-1,\alpha) = -1-\alpha$ for every $\alpha\in [0,1]$. Thus
	$$
	O(B;\mathcal{X}) = \sup_{(u,\alpha)\in\mathbb{S}^{0}\times [0,1]} |s_{B}(u,\alpha) - 2| = \sup_{\alpha\in[0,1]} |1-\alpha| = 1
	$$
	Thus $D_{FP}(B;\mathcal{X}) = 1/2 < D_{FP}(A;\mathcal{X})$ but $A,B\in\Med_1(\mathcal X)=\Med_s(\mathcal X)$.
\end{example}

\begin{theorem}\label{teoremaProjection1mediana}
	Let $\mathcal{X}\in L^{0}[\mathcal{F}_{c}(\mathbb{R})]$ be  non-degenerate and such that the median of $s_{\mathcal{X}}(u,\alpha)$ is unique for each $u\in\mathbb{S}^{0}$ and $\alpha\in [0,1]$. Then
	$$\Med_s(\mathcal X)= \Med(\mathcal X;D_{FP})=\{\med_{Si}(\mathcal X)\}=\{\med_{Gr}(\mathcal X)\}.$$
\end{theorem}

Let us denote by $C^{0}[\mathcal{F}_{c}(\mathbb{R}^{p})]$ the set of all fuzzy random variables $\mathcal{X}\in L^{0}[\mathcal{F}_{c}(\mathbb{R}^p)]$ such that $s_{\mathcal{X}}(u,\alpha)$ is a continuous random variable for each $(u,\alpha)\in\mathbb{S}^{p-1}\times [0,1]$. 
As shown in \cite{simplicialsegundo}, fuzzy generalizations of simplicial depth have nice properties when computed with respect to  $\mathcal{X}\in C^{0}[\mathcal{F}_{c}(\mathbb{R}^{p})].$ In that situation, there is a unique support median that maximizes both the modified simplicial depth and  the fuzzy simplicial depth, in \eqref{ms} and \eqref{fs} respectively.

\begin{theorem}\label{teoremaSimplicial1mediana}
	Let $\mathcal{X}\in C^{0}[\mathcal{F}_{c}(\mathbb{R})]$. Then
	$$\Med_s(\mathcal X)=\Med(\mathcal X;D_{mS})=\Med(\mathcal X;D_{FS})=\{\med_{Si}(\mathcal X)\}=\{\med_{Gr}(\mathcal X)\}.$$
\end{theorem}

If the fuzzy random variable $\mathcal X\not\in C^{0}[\mathcal{F}_{c}(\mathbb{R})]$, both $\Med_s(\mathcal X)\neq \Med(\mathcal X;D_{FS})$ and $\Med_s(\mathcal X)\neq \Med(\mathcal X;D_{mS})$ can happen, as shown by the following example.

\begin{example}\label{ejemplosimpli}
	Let $X$ be a real random variable with cumulative distribution function
	\begin{equation}\nonumber
		\mathbb{F}_{X}(x) := \left\{ \begin{array}{lcr}
			0,  & \text{ if } x\in (-\infty,0) 
			\\ 
			.145x,  &  \text{ if } x\in [0,2) 
			\\ 
			.01x + .47, & \text{ if } x\in [2,53)
			\\ 
			1, & \text{ if } x\in [53,\infty).
		\end{array}
		\right.
	\end{equation}
Note that %
$\Med(X) = \{3\},$ which implies, setting $\mathcal{X} = \text{I}_{\{X\}}\in L^0[\mathcal{F}_{c} (\mathbb{R}^p)]$,
\begin{equation}\label{s3}
\Med_s(\mathcal{X}) = \left\{\text{I}_{\{3\}}\right\}.
\end{equation}
 Let us consider the simplicial depth based on $L^0[\mathcal{F}_{c} (\mathbb{R}^p)]$ and $\mathcal{F}_{c}(\mathbb{R})$. Using \citep[Proposition 5.1]{simplicialsegundo} (reproduced as Lemma \ref{integrandoDFS} below), 
$$
D_{FS}\left(\text{I}_{\{x\}};\mathcal{X}\right) = 1 - (1-\mathbb{F}_{X}(x))^2 - (\mathbb{F}_{X}(x) - \mathbb{P}(X = x))^2,
$$
and, therefore,
$$
D_{FS}(\text{I}_{\{3\}};\mathcal{X}) = .5
\mbox{ 
and
}
D_{FS}(\text{I}_{\{2\}};\mathcal{X}) = .6558.
$$
Taking into account \eqref{s3}, that implies  $\Med_s(\mathcal{X})\neq\Med(\mathcal{X};D_{FS})$ because $I_{\{3\}}$ does not maximize $D_{FS}$. The case of $D_{mS}$ is done analogously.
\end{example}

\begin{remark}
	One of the main uses of depth functions is to generalize the concept of a median as the element with maximal depth. It follows from the results in this section that, for $p = 1$, the confluence of the notions of 1-median, support median and depth median for the depth functions $D_{FT}$, $D_{FP}$, $D_{1}$, $D_{mS}$ and $D_{FS}$, is nice and natural. The fact that support medians may not exist in the case $p>1$ underlines that depth medians are a reasonable approach to defining medians for multivariate fuzzy data. 

In our study of the case $p=1$, in some cases we have resorted to additional assumptions which seem to be reasonable. For example, the definition of projection depth presupposes that the univariate median is taken to be the midpoint of all possible medians, whereas the definition of a support median allows for arbitrary medians; studying the relationship between both concepts under the assumption of a unique median is natural. In the case of simplicial depths, they are generalizations of multivariate simplicial depth which enjoys good properties for continuous variables, whence restricting ourselves to $C^0[\pfc(\R^p)]$ is natural.
\end{remark}

Next we present two examples in which the set of medians of a fuzzy random variable is computed explicitly. 

\begin{example}
	Let $(\{\omega_{1},\omega_{2},\omega_{3},\omega_{4}\}, \mathcal{P}(\{\omega_{1},\omega_{2},\omega_{3},\omega_{4}\}),\mathbb{P})$ be a probabilistic space such that $\mathbb{P}(\omega_{1}) = \mathbb{P}(\omega_{2}) = \mathbb{P}(\omega_{3}) = \mathbb{P}(\omega_{4})$. Let $\mathcal{X} : \Omega\rightarrow\mathcal{F}_{c}(\mathbb{R})$ be a fuzzy random variable defined by $\mathcal{X}(\omega_{1}) = T(1,2,3)$, $\mathcal{X}(\omega_{2}) = \text{I}_{\{4\}}$, $\mathcal{X}(\omega_{3}) = \text{I}_{\{5\}}$ and $\mathcal{X}(\omega_{4}) = T(6,6,7)$. By Theorem \ref{Teorema1Mediana}, the set of $1$-medians is the set of support medians:
	$$
	\text{Med}_1(\mathcal X) = \{A\in\mathcal F_c(\mathbb R) : s_A(u,\alpha)\in\text{Med}(s_{\mathcal X}(u,\alpha)) \text{ for every } u\in\mathbb S^0,\alpha\in [0,1]\}
	$$
	
	The support function of $\mathcal X(\omega_i)$, with $i\in\{1,2,3,4\}$, has the following expressions:
	
	\begin{equation}\nonumber
		s_{\mathcal X(\omega_1)} (u,\alpha)  = \left\{ \begin{array}{lcc}
			3-\alpha  & \text{ if } u = 1, \alpha\in [0,1], \\
			\\
			-1-\alpha &  \text{ if } u = -1,\alpha\in [0,1] \\
		\end{array}
		\right.
	\end{equation}

	\vspace{-1cm}
	
	\begin{equation}\nonumber
	s_{\mathcal X(\omega_2)} (u,\alpha) = 4u
\end{equation}

	\vspace{-0.3cm}

\begin{equation}\nonumber
	s_{\mathcal X(\omega_3)} (u,\alpha) = 5u
\end{equation}
	
	and
	
	\begin{equation}\nonumber
		s_{\mathcal X(\omega_4)} (u,\alpha)  = \left\{ \begin{array}{lcc}
			7-\alpha  & \text{ if } u = 1, \alpha\in [0,1], \\
			\\
			-6 &  \text{ if } u = -1,\alpha\in [0,1] \\
		\end{array}
		\right.
	\end{equation}
	
	Therefore
	$$
	\text{Med}(s_{\mathcal X}(1,\alpha)) = \text{Med}(\{3-\alpha, 4,5,7-\alpha\}) = [4,5]
	$$
	and
	$$
	\text{Med}(s_{\mathcal X}(-1,\alpha)) = \text{Med}(\{-1-\alpha, -4,-5,-6\}) = [-5,-4]
	$$
	
	Thus, the set of $1$-medians is
	$$
	\text{Med}_1(\mathcal X) = \{M\in\mathcal{F}_{c}(\mathbb{R}) : M_\alpha\subseteq [4,5] \text{ for all }\alpha\in [0,1]\}
	$$

	By Theorem \ref{teoremaTukey1mediana}, the above set is also the set of maximizers of $D_{FT}$. It is graphically represented in Figure \ref{figura1simetria}.
\end{example}

\begin{figure}[h]
	\centering
	\includegraphics[scale=0.5]{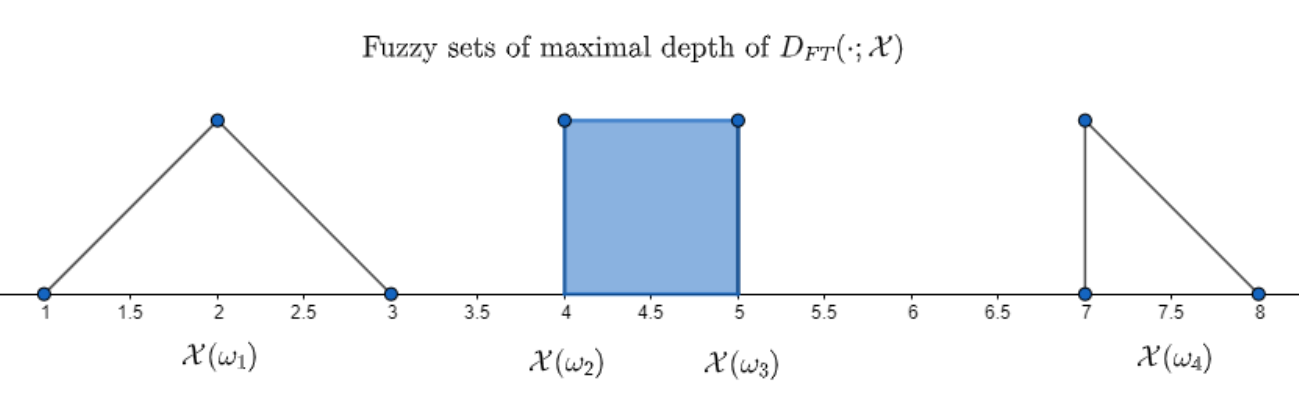}
	\caption{Visual representation of the set of fuzzy sets maximizing $D_{FT}(\cdot;\mathcal{X})$. Every fuzzy set `contained' in the blue area has maximal depth.}
\label{figura1simetria}
\end{figure}

\begin{example}
	Let $\Omega = \{\omega_{1},\omega_{2},\omega_{3},\omega_{4}\}$ and $(\Omega, \mathcal{P}(\Omega),\mathbb{P})$ be a probabilistic space such that $\mathbb{P}(\omega_{1}) = \mathbb{P}(\omega_{2}) = \mathbb{P}(\omega_{3}) = \mathbb{P}(\omega_{4})$. Let $\mathcal{X} : \Omega\rightarrow\mathcal{F}_{c}(\mathbb{R})$ be a fuzzy random variable defined by $\mathcal{X}(\omega_{1}) = T(1,2,3)$, $\mathcal{X}(\omega_{2}) = T(4,4,5)$, $\mathcal{X}(\omega_{3}) = T(6,7,8)$ and $\mathcal{X}(\omega_{4}) = \text{I}_{\{9\}}$.
\end{example}

Let us compute the set of support medians. For that, we compute the $\alpha$-levels for each $\mathcal{X}(\omega_{i}),$ $i=1, \ldots, 4,$ obtaining, for any $\alpha\in[0,1],$
$(\mathcal{X}(\omega_{1}))_{\alpha} = [\alpha + 1 , 3 - \alpha],$ 
		$(\mathcal{X}(\omega_{2}))_{\alpha} = [4, 5-\alpha],$ 
		$(\mathcal{X}(\omega_{3}))_{\alpha} = [\alpha + 6, 8 - \alpha]\mbox{ and }
(\mathcal{X}(\omega_{4}))_{\alpha} = \{9\}.$
Then the support functions of those fuzzy sets are
\begin{equation}\nonumber
	s_{\mathcal X(\omega_1)} (u,\alpha)  = \left\{ \begin{array}{lcc}
		3-\alpha  & \text{ if } u = 1, \alpha\in [0,1], \\
		%\\
		-1-\alpha &  \text{ if } u = -1,\alpha\in [0,1] %\\
	\end{array}
	\right.
\end{equation}
\begin{equation}\nonumber
	s_{\mathcal X(\omega_2)} (u,\alpha)  = \left\{ \begin{array}{lcc}
		5-\alpha  & \text{ if } u = 1, \alpha\in [0,1], \\
		%\\
		-4 &  \text{ if } u = -1,\alpha\in [0,1] %\\
	\end{array}
	\right.
\end{equation}
\begin{equation}\nonumber
	s_{\mathcal X(\omega_3)} (u,\alpha)  = \left\{ \begin{array}{lcc}
		8-\alpha  & \text{ if } u = 1, \alpha\in [0,1], \\
		%\\
		-6-\alpha &  \text{ if } u = -1,\alpha\in [0,1] %\\
	\end{array}
	\right.
\end{equation}
and
$
s_{\mathcal X(\omega_4)} (u,\alpha) = 9u.
$
Consequently, for all $\alpha\in [0,1],$
$$
\text{Med}(s_{\mathcal X}(1,\alpha)) = \text{Med}(\{3-\alpha,5-\alpha,8-\alpha,9\}) = [5-\alpha,8-\alpha]
$$
and
$$
\text{Med}(s_{\mathcal X}(-1,\alpha)) = \text{Med}(\{-1-\alpha,-4,-6-\alpha,-9\}) = [-4,-6-\alpha].
$$
Thus,  the set of support medians, which coincides with the 1-medians (Theorem \ref{Teorema1Mediana}) and the Tukey medians (Theorem \ref{teoremaTukey1mediana}) is
\begin{equation}\nonumber
	\Med_1(\mathcal X) = \{A\in\mathcal{F}_{c}(\mathbb{R}) : \inf A_{\alpha}\in [4, 6 + \alpha], \sup A_{\alpha}\in [5-\alpha, 8 - \alpha] \text{ for all } \alpha\in [0,1]\},
\end{equation}
displayed in Figure \ref{figura2simetria}.

\begin{figure}[h]
	\centering
	\includegraphics[scale=0.5]{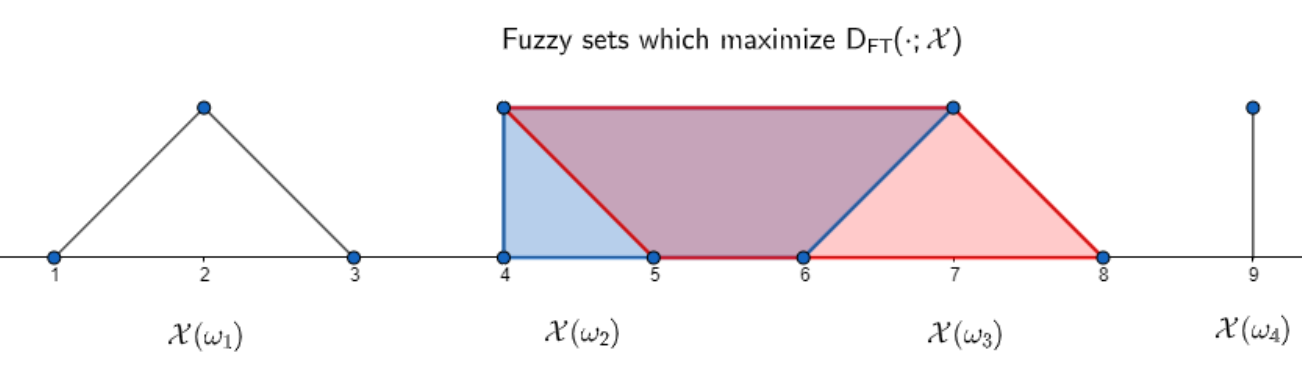}
	\caption{Illustration (in blue and red) of the set of medians of $\mathcal X$. The areas where the left and right slope of a median must be contained are represented in different colors.}
	\label{figura2simetria}
\end{figure}

\section{Proofs}\label{proofs}

\begin{proof}[Proof of Proposition \ref{gregorsupport}]
	Let $\mathcal{X}$ be a fuzzy random variable in $\R$ and let $A\in\Med_{s}(\mathcal{X})$ be a support median of $\mathcal{X}$. The $\alpha$-level of $A$ can be expressed in terms of its support function as
	$$
	A_{\alpha} = [-s_{A}(-1,\alpha), s_{A}(1,\alpha)] = [\inf A_{\alpha},\sup A_{\alpha}]
	$$
	whenever $\alpha\in [0,1]$.
	By definition $s_{A}(u,\alpha)\in\Med (s_{\mathcal{X}}(u,\alpha))$ whenever $u\in\mathbb{S}^{0}, \alpha\in [0,1]$. Thus
	\begin{equation}\nonumber
		\begin{aligned}
			-s_{A}(-1,\alpha)&\in\Med(-s_{\mathcal{X}}(-1,\alpha)) = \Med(\inf \mathcal{X}_{\alpha})\\
			s_{A}(1,\alpha)&\in\Med(s_{\mathcal{X}}(1,\alpha)) = \Med(\sup \mathcal{X}_{\alpha})
		\end{aligned}
	\end{equation}
	
	Let $\med_{Gr}(\mathcal{X})$ be the fuzzy median of $\mathcal{X}$ in  \eqref{medianaGregor}. By \eqref{propGregor}, 
$$\underline{\med}(\inf\mathcal{X}_{\alpha})\leq -s_{A}(-1,\alpha)\mbox{ and }s_{A}(1,\alpha)\leq \overline{\med}(\sup\mathcal{X}_{\alpha}).$$ Thus $A_{\alpha}\subseteq(\med_{Gr}(\mathcal{X}))_{\alpha}$ for every $\alpha\in [0,1]$. Consequently, the support median $A$ is contained in $\med_{Gr}(\mathcal{X})$.
\end{proof}

The next characterization is used for the proof of Theorem \ref{Teorema1Mediana}.

\begin{lemma}[\cite{ResultadoMedianaR}]\label{teoremamediana}
	Let $X$ be a real random variable and $x\in\mathbb{R}$. Then, $x\in \Med(X)$ if and only if
	$$\text{E}[|X - y|] = \inf_{z\in\mathbb{R}}\text{E}[|X - z|]$$
\end{lemma}

\begin{proof}[Proof of Theorem \ref{Teorema1Mediana}]
	Let $\mathcal{X}$ be a fuzzy random variable. 	The proof of the result in \eqref{teoremaBea} (\cite[Theorem 4.1]{medianfuzzy1}) already shows that $\Med_s(\mathcal X)\subseteq\Med_1(\mathcal X)$ for every fuzzy random variable $\mathcal{X}$. That is so because the proof that $\med_{Si}(\mathcal X)$ is a $1$-median only uses the fact that the endpoints of each $\alpha$-cut are medians. Notice that $\inf \mathcal X_\alpha=-s_{\mathcal X}(-1,\alpha)$ and $\sup \mathcal X_\alpha=s_{\mathcal X}(1,\alpha)$, so reasoning with the endpoints is equivalent to reasoning with the support function.

	There remains to prove $\Med_1(\mathcal X)\subseteq\Med_s(\mathcal X)$. Reasoning by contradiction, assume there exists $B\in \Med_1(\mathcal X)\backslash \Med_s(\mathcal X)$. Then the set
	$$K = \{(u,\alpha)\in\mathbb{S}^{0}\times [0,1] : s_{B}(u,\alpha)\not\in \Med (s_{\mathcal{X}}(u,\alpha))\}$$
	is not empty. We consider the measure $\upsilon\otimes\nu$ over the set $\mathbb{S}^{0}\times [0,1]$, where $\upsilon$ denotes the uniform distribution over $\mathbb{S}^{0}$ and $\nu$ denotes the Lebesgue measure over $[0,1]$.
	
	Fix any $A\in \Med_s(\mathcal X)$. Clearly,
	$$\text{E}[\rho_{1}(\mathcal{X}, B)] = \text{E}[\rho_{1}(\mathcal{X},A)]$$
	
	Using Fubini's theorem, which is possible due to the joint measurability of $s_{\mathcal X(\cdot)}(\cdot,\cdot)$ in its three arguments \cite[Proposition 4.4]{jointPedroMiriam},
	
	\begin{equation}\label{ecuacion1resultadomedianas}
		\int_{\mathbb{S}^{0}}\int_{[0,1]}\text{E}[|s_{\mathcal{X}}(u,\alpha) - s_{B}(u,\alpha)|] dud\alpha = \int_{\mathbb{S}^{0}}\int_{[0,1]}\text{E}[|s_{\mathcal{X}}(u,\alpha) - s_{A}(u,\alpha)|] dud\alpha
	\end{equation}
	
	By Lemma \ref{teoremamediana}, if $x\in\mathbb{R}$ is not a median of a real random variable $X$, then $\text{E}[|X - x|] > \inf_{y}\text{E}[|X - y|]$. Using \eqref{ecuacion1resultadomedianas}, we deduce that $K$ has $(\upsilon\otimes\nu)$-measure $0$, that is $(\upsilon\otimes\nu) (K) = 0$. 
	Set $H = \mathbb{S}^{0}\times [0,1]\setminus K$. Thus $H$ has $(\upsilon\otimes\nu)$-measure $1$ and $H$ is a dense subset of $\mathbb{S}^{0}\times [0,1]$. As $\mathbb{S}^{0} = \{-1,1\}$, we can express $H$ as the union of  $\{1\}\times U_{1}$ and $\{-1\}\times U_{-1}$, where $U_{1}$ and $U_{-1}$ are dense subsets of $[0,1]$.
	
	Since $K$ is non-empty, there is some $(u,\alpha)\in K$. Without loss of generality, we assume $u = 1$ (if $u=-1$ the argument is analogous or, alternatively, the same argument applies to $-\mathcal X$). We distinguish two cases, at the end of each we will reach a contradiction.
	
	\textbf{Case I.} ($\alpha > 0$) As $U_{1}$ is dense in $[0,1]$, there exists a non-decreasing sequence $\{\alpha_{n}\}_{n}$ of elements of $U_{1}$ such that $(1,\alpha_{n})\in H$ for all $n\in\mathbb{N}$ and $\lim_{n}\alpha_{n} = \alpha$. As $\{\alpha_{n}\}_{n}$ is non-decreasing,  $\{s_{B}(1,\alpha_{n})\}_{n}$ is a non-increasing sequence with $s_{B}(1,\alpha_{n})\geq s_{B}(1,\alpha)$ for all $n\in\mathbb{N}$. By \cite{MingSupport}, $s_{U}(\cdot,\alpha)$ is left-continuous  for all $\alpha\in [0,1]$ and for all $U\in\mathcal{F}_{c}(\mathbb{R})$. Thus $\lim_{n} s_{B}(1,\alpha_{n}) = s_{B}(1,\alpha)$. Analogously, $\lim_{n} s_{\mathcal{X}(\omega)}(1,\alpha_{n}) = s_{\mathcal{X}(\omega)}(1,\alpha)$ for each $\omega\in\Omega$.
	
	First,
	\begin{equation}
		\begin{aligned}\nonumber
			&\mathbb{P}(s_{\mathcal{X}}(1,\alpha) < s_{B}(1,\alpha)) \leq \lim\inf_{n}\mathbb{P}(s_{\mathcal{X}}(1,\alpha_{n}) < s_{B}(1,\alpha))\leq \\ \nonumber
			&\lim\inf_{n}\mathbb{P}(s_{\mathcal{X}}(1,\alpha_{n}) < s_{B}(1,\alpha_{n})) = \lim\inf_{n} \left(1 - \mathbb{P}(s_{\mathcal{X}}(1,\alpha_{n})\geq s_{B}(1,\alpha_{n}))\right)\leq \cfrac{1}{2},
		\end{aligned}
	\end{equation}
	where the second inequality is due to $s_{B}(1,\alpha)\leq s_{B}(1,\alpha_{n})$ for all $n\in\mathbb{N}$ and third inequality to $(1,\alpha_{n})\in H$, thus $s_{B}(1,\alpha_{n})\in\Med(s_{\mathcal{X}}(1,\alpha_{n}))$ for all $n\in\mathbb{N}$. Hence $\mathbb{P}(s_{\mathcal{X}}(1,\alpha)\geq s_{B}(1,\alpha))\geq 1/2$.
	
	Denoting by $F_{u,\alpha}$ the cumulative distribution function of the real random variable $s_{\mathcal{X}}(u,\alpha)$, 
	\begin{equation}\nonumber
		\cfrac{1}{2}\leq\mathbb{P}(s_{\mathcal{X}}(1,\alpha_{n})\leq s_{B}(1,\alpha_{n}))\leq \mathbb{P}(s_{\mathcal{X}}(1,\alpha)\leq s_{B}(1,\alpha_{n})) = F_{1,\alpha}(s_{B}(1,\alpha_{n})),
	\end{equation}
	where the first inequality is due to $(1,\alpha_{n})\in H$ and second one holds because $s_{\mathcal{X}(\omega)}(1,\alpha_{n})\geq s_{\mathcal{X}(\omega)}(1,\alpha)$ for all $n\in\mathbb{N}$ and $\omega\in\Omega$. Taking limits at both sides,
	\begin{equation}\nonumber
		\cfrac{1}{2}\leq \lim_{n\rightarrow\infty} F_{1,\alpha}(s_{B}(1,\alpha_{n})) = F_{1,\alpha}(s_{B}(1,\alpha)),
	\end{equation}
	where the equality is due to the continuity from the right of the cumulative distribution function. Thus $s_{B}(1,\alpha)\in\Med(s_{\mathcal{X}}(1,\alpha))$ but $(1,\alpha)\in K$, a contradiction.
	
	\textbf{Case II.}($\alpha = 0$) As $U_{1}$ is dense in $[0,1]$, there exists a non-increasing sequence $\{\alpha_{n}\}_{n}$ such that $(1,\alpha_{n})\in H$ for all $n\in\mathbb{N}$ and $\lim_{n}\alpha_{n} = 0$. Then, the sequence $\{s_{U}(\alpha_{n})\}_{n}$ is non-decreasing and such that $s_{U}(1,\alpha_{n})\leq s_{U}(1,0)$ for all $n\in\mathbb{N}$ and for all $U\in\mathcal{F}_{c}(\mathbb{R})$. By \cite{MingSupport}, the function $s_{U}(\cdot,\alpha)$ is right-continuous at $\alpha = 0$, thus
	\begin{equation}\nonumber
		\lim_{n\rightarrow\infty} s_{U}(1,\alpha_{n}) = s_{U}(1,0),
	\end{equation}
	for all $U\in\mathcal{F}_{c}(\mathbb{R})$.
	
	We prove that $s_{B}(1,0)\in\Med(s_{\mathcal{X}}(1,0))$. First,
	\begin{equation}
		\begin{aligned}\nonumber
			&\mathbb{P}(s_{\mathcal{X}}(1,0) > s_{B}(1,0))\leq\lim\inf_{n}\mathbb{P}(s_{\mathcal{X}}(1,\alpha_{n}) > s_{B}(1,0))\leq \\ \nonumber
			&\mathbb{P}(s_{\mathcal{X}}(1,\alpha_{n}) > s_{B}(1,\alpha_{n})) = \lim\inf_{n} \left(1 - \mathbb{P}(s_{\mathcal{X}}(1,\alpha_{n}))\right)\leq s_{B}(1,\alpha_{n}))\leq \cfrac{1}{2},
		\end{aligned}
	\end{equation}
	where the second inequality is due to $s_{B}(1,\alpha_{n})\leq s_{B}(1,0)$ for all $n\in\mathbb{N}$ and the third inequality to the fact that $(1,\alpha_{n})\in H$. This implies that $\mathbb{P}(s_{\mathcal{X}}(1,0)\leq s_{B}(1,0))\geq 1/2$.
	
	On the other hand, 
	\begin{equation}
		\begin{aligned}\nonumber
			&\cfrac{1}{2}\leq\mathbb{P}(s_{\mathcal{X}}(1,\alpha_{n})\geq s_{B}(1,\alpha_{n}))\leq\mathbb{P}(s_{\mathcal{X}}(1,0)\geq s_{B}(1,\alpha_{n})) = \\ \nonumber
			&\mathbb{P}(s_{\mathcal{X}}(1,0)\geq s_{B}(1,0)) + \mathbb{P}(s_{\mathcal{X}}(1,0)\in (s_{B}(1,\alpha_{n}), s_{B}(1,0)]),
		\end{aligned}
	\end{equation}
	where the first inequality is due to $(1,\alpha_{n})\in H$ and the second holds because $s_{U}(1,\alpha_{n})\leq s_{U}(1,0)$ for all $U\in\mathcal{F}_{c}(\mathbb{R})$ and for all $n\in\mathbb{N}$. If we take limits, it is clear that
	\begin{equation}\nonumber
		\lim_{n\rightarrow\infty}\mathbb{P}(s_{\mathcal{X}}(1,0)\in (s_{B}(1,\alpha_{n}), s_{B}(1,0)]) = 0,
	\end{equation}
	because $\lim_{n} s_{B}(1,\alpha_{n}) = s_{B}(1,0)$. Thus $\mathbb{P}(s_{\mathcal{X}}(1,0)\geq s_{B}(1,0))\geq 1/2$ and $s_{B}(1,0)\in\Med(s_{\mathcal{X}}(1,0))$. That implies $(1,0)\not\in K$, a contradiction.
	
	We conclude from this that $\Med_1(\mathcal X)\backslash \Med_s(\mathcal X)$ is empty, which completes the proof.
\end{proof}

	The following property, which relies on the fact that a convex combination of $x,y\in\R$ is always between $x$ and $y$, is used in the proof of Theorem \ref{D1rho1}.
	
	\begin{lemma}\label{lema1}
		Let $x,y\in\mathbb{R}$ and $\lambda\in [0,1]$. Then
		$$		\cfrac{1}{1+\lambda}\cdot(x + \lambda\cdot y)\in [\min\{x,y\},\max\{x,y\}].	$$
	\end{lemma}
	
	\begin{proof}[Proof of Theorem \ref{D1rho1}]
		Let $A,B,C\in\mathcal{F}_{c}(\mathbb{R})$ be fuzzy sets such that
		\begin{equation}\label{ecuacion1}
			\rho_{1}(A,B) = \rho_{1}(A,C) + \rho_{1}(B,C)
		\end{equation}
		and $A$ maximizes $D_{1}$. By  \eqref{teoremaBea}, such $A$ indeed exists and using Theorem \ref{Teorema1Mediana}, $s_{A}(u,\alpha)\in\text{Med}(s_{\mathcal{X}}(u,\alpha))$ for every $u\in\mathbb{S}^{0}$ and $\alpha\in [0,1]$. Thus  
		\begin{equation}\label{minimomediana}
			s_{A}(u,\alpha)\in\arg\min_{y\in\mathbb{R}}\text{E}[|y - s_{\mathcal{X}}(u,\alpha)|]
		\end{equation}
		for all $u\in\mathbb{S}^{0}$ and $\alpha\in [0,1]$.
		By \eqref{ecuacion1},
		
		\begin{equation}
			\begin{aligned}\nonumber	
				&\int_{\mathbb{S}^{p-1}}\int_{[0,1]} |s_{A}(u,\alpha) - s_{B}(u,\alpha)|dud\alpha = \\ \nonumber &\int_{\mathbb{S}^{p-1}}\int_{[0,1]}(|s_{A}(u,\alpha) - s_{C}(u,\alpha)| + |s_{B}(u,\alpha) - s_{C}(u,\alpha)|) dud\alpha.
			\end{aligned}
		\end{equation}
		Thus
		$$
		\int_{\mathbb{S}^{p-1}}\int_{[0,1]}  (|s_{A}(u,\alpha) - s_{C}(u,\alpha)| + |s_{B}(u,\alpha) - s_{C}(u,\alpha)| - |s_{A}(u,\alpha) - s_{B}(u,\alpha)|) dud\alpha = 0
		$$
		
		The integrand in the last expression is non-negative due to the triangle inequality. Thus, there exists a dense subset $D\subseteq\mathbb{S}^{p-1}\times [0,1]$ with $(\upsilon\otimes\nu)$-measure $1$  such that
		$$
		|s_{A}(u,\alpha) - s_{B}(u,\alpha)| = |s_{A}(u,\alpha) - s_{C}(u,\alpha)| + |s_{B}(u,\alpha) - s_{C}(u,\alpha)|
		$$
		for every $(u,\alpha)\in D$. By the necessary conditions for the triangle inequality to be an equality, for every $(u,\alpha)\in D$ there exists $\lambda_{u,\alpha}\in [0,1]$ such that
		$$
		s_{A}(u,\alpha) - s_{C}(u,\alpha) = \lambda_{u,\alpha}\cdot (s_{C}(u,\alpha) + s_{B}(u,\alpha))
		$$
		After some algebra,
		\begin{equation}\label{ecuacionC}
			s_{C}(u,\alpha) = \cfrac{1}{1+\lambda_{u,\alpha}}\cdot(s_{A}(u,\alpha) + \lambda_{u,\alpha}\cdot s_{B}(u,\alpha))
		\end{equation}
		Setting $r(u,\alpha) = 1/(1+\lambda_{u,\alpha})$, 
		\begin{equation}
			\begin{aligned}\nonumber
				&\text{E}[\rho_{1}(C,\mathcal{X})] = \int_{\mathbb{S}^{0}\times [0,1]}\text{E}[|s_{C}(u,\alpha) - s_{\mathcal{X}}(u,\alpha)|]dud\alpha = \\ \nonumber
				&\int_{\mathbb{S}^{0}\times [0,1]}\text{E}[|r(u,\alpha)\cdot s_{A}(u,\alpha) + (1-r(u,\alpha))\cdot s_{B}(u,\alpha) - s_{\mathcal{X}}(u,\alpha)|]dud\alpha  = \\ \nonumber
				&\int_{\mathbb{S}^{0}\times [0,1]}\text{E}[|r(u,\alpha)\cdot(s_{A}(u,\alpha) - s_{\mathcal{X}}(u,\alpha)) + (1-r(u,\alpha))\cdot (s_{B}(u,\alpha) - s_{\mathcal{X}}(u,\alpha))|]dud\alpha\leq\\ \nonumber
				&\int_{\mathbb{S}^{0}\times [0,1]}r(u,\alpha)\cdot\text{E}[|s_{A}(u,\alpha) - s_{\mathcal{X}}(u,\alpha)|] + (1-r(u,\alpha))\cdot\text{E}[|s_{B}(u,\alpha) - s_{\mathcal{X}}(u,\alpha)|] dud\alpha\leq\\ \nonumber
				&\int_{\mathbb{S}^{0}\times [0,1]}\text{E}[|s_{B}(u,\alpha) - s_{\mathcal{X}}(u,\alpha)|]dud\alpha = \text{E}[\rho_{1}(B,\mathcal{X})]
			\end{aligned}
		\end{equation}
		Since the function $s_{\mathcal{X}}$ is jointly measurable in its three arguments \cite[Lemma 4]{Kra} the first identity is due to Fubini's theorem and \ref{rhor}. The second one is due to \eqref{ecuacionC}, the second inequality to the triangle inequality, and the last identity to \eqref{minimomediana}. Thus $D_{1}(C;\mathcal{X})\leq D_{1}(B;\mathcal{X})$.
	\end{proof}

The next result is known and characterizes the set of medians of a real random variable as the numbers which maximize Tukey depth. It is very useful for the proof of Theorem \ref{teoremaTukey1mediana}.

\begin{lemma}\label{teoremasupremoTukey}
	Let $X$ be a real random variable. Then, $x\in\med(X)$ if, and only if, $x\in\arg\max_{y\in\mathbb{R}} HD(y;X)$.
\end{lemma}

\begin{proof} 
	Let $X$ be a real random variable, and $x\in\mathbb{R}$ such that $x\in\arg\max_{y\in\mathbb{R}} HD(y;X)$. We assume by contradiction that $x\not\in\Med(X)$. Without loss of generality, we suppose that $\mathbb{P}(X\leq x) < 1/2$. Let $z\in\Med(X)$. Then
	\begin{equation}
		HD(x;X) = \min\{\mathbb{P}(X\leq x), \mathbb{P}(X\geq x)\} < \cfrac{1}{2}\leq HD(z;X),
	\end{equation}
	which is a contradiction.
	
	Now let $x\in\Med(X)$. We must prove that $x\in\arg\max_{y\in\mathbb{R}} HD(y;\mathbb{R})$. Consider $x_{0} = \inf\{y\in\mathbb{R} : \mathbb{P}(X\leq y)\geq 1/2\}$ and $x_{1} = \inf\{y\in\mathbb{R} : \mathbb{P}(X\geq y)\geq 1/2\}$. By definition, $\mathbb{P}(X\leq x_{0})\geq 1/2$ and $\mathbb{P}(X\geq x_{1})\geq 1/2$. The interval $[x_{0},x_{1}]$ is the set of medians of $X$. If $x_{0} = x_{1}$, the median of $X$ is unique and the proof is done. If $x_{0} < x_{1}$, the median is not unique,	but $\mathbb{P}(X\in (x_{0},x_{1})) = 0$ and every median has the same halfspace depth. Hence, $x\in\arg\max_{y\in\mathbb{R}} HD(y;X)$.
\end{proof}

\begin{proof}[Proof of Theorem \ref{teoremaTukey1mediana}]
	Let $A\in\Med_s(\mathcal X)$, i.e., $s_{A}(u,\alpha)\in\Med(s_{\mathcal{X}}(u,\alpha))$ for all $u\in\mathbb{S}^{0}$ and $\alpha\in [0,1]$. Let us show $A\in\Med(\mathcal X, D_{FT})$, that is, $A$ maximizes $D_{FT}(\cdot;\mathcal{X})$. By \citet{TukeySMPS}, we can express $D_{FT}(A;\mathcal{X})$ as 
	\begin{equation}\nonumber
		D_{FT}(A;\mathcal{X}) = \inf_{u,\alpha} HD(s_{A}(u,\alpha),s_{\mathcal{X}}(u,\alpha)),
	\end{equation}
	where $HD$ is the Tukey depth for the multivariate case ($\mathbb{R}$ in this case). For any arbitrary $B\in\mathcal{F}_{c}(\mathbb{R})$,
	\begin{equation}
		\begin{aligned}\nonumber
			D_{FT}(A;\mathcal{X}) &= \inf_{u,\alpha} HD(s_{A}(u,\alpha),s_{\mathcal{X}}(u,\alpha)) = \inf_{u,\alpha}\sup_{y\in\mathbb{R}} HD(y;s_{\mathcal{X}}(u,\alpha))\geq \\ \nonumber
			&\sup_{y\in\mathbb{R}}\inf_{u,\alpha} HD(y;s_{\mathcal{X}}(u,\alpha))\geq \inf_{u,\alpha} HD(s_{B}(u,\alpha);s_{\mathcal{X}}(u,\alpha))
			=D_{FT}(B;\mathcal X),
		\end{aligned}
	\end{equation}
	where the second identity uses Lemma \ref{teoremasupremoTukey} and the fact that $A\in\Med_s(\mathcal X)$. That proves $\mathcal M_{\mathcal X}\subset \Med(\mathcal X;D_{FT})$.

	For the converse inclusion, let $A\in \Med(\mathcal X;D_{FT})$. Reasoning by contradiction, assume there exist $u\in\mathbb{S}^{0}, \alpha\in [0,1]$ such that $s_{A}(u,\alpha)\not\in\Med(s_{\mathcal{X}}(u,\alpha))$. Without loss of generality, assume
	$$\mathbb{P}(s_{\mathcal{X}}(u,\alpha)\leq s_{A}(u,\alpha)) < 1/2.$$
	From the definition of Tukey depth for fuzzy sets, $D_{FT}(A;\mathcal{X}) < 1/2$. By  \eqref{teoremaBea}, there exists $\med_{Si}(\mathcal{X})\in\Med_{1}(\mathcal{X}) $ such that $s_{\med_{Si}}(u,\alpha)\in\Med (s_{\mathcal{X}}(u,\alpha))$ whenever $u\in\mathbb{S}^{0}$, $\alpha\in [0,1]$. Using the first part of the proof, $D_{FT}(\med_{Si}(\mathcal{X});\mathcal{X})\geq 1/2$. Thus
	\begin{equation}\nonumber
		D_{FT}(A;\mathcal{X})  < 1/2\leq D_{FT}(\med_{Si}(\mathcal{X});\mathcal{X}),
	\end{equation}
	yields to a contradiction to the fact that $A$ maximizes $D_{FT}(\cdot;\mathcal X)$.
\end{proof}

\begin{proof}[Proof of Theorem \ref{teoremaProjection1mediana}]
	By the assumption, the support median is unique and so
	$$\Med_s(\mathcal X)=\{\med_{Si}(\mathcal X)\}=\{\med_{Gr}(\mathcal X)\}.$$
	Since $\mathcal{X}$ is non-degenerate, $\text{MAD}(s_{\mathcal{X}}(u,\alpha)) > 0$ whenever $u\in\mathbb{S}^{0}$, $\alpha\in [0,1]$. Thus $$O(\med_{Si}(\mathcal{X});\mathcal{X}) = 0$$
	As $O(U;\mathcal{X})\geq 0$ for all $U\in\mathcal{F}_{c}(\mathbb{R})$,
	$$1 = D_{FP}(\med_{Si}(\mathcal{X});\mathcal{X}) = \sup_{U\in\mathcal{F}_{c}(\mathbb{R})} D_{FP}(U;\mathcal{X}).$$
	That proves $\{\med_{Si}(\mathcal{X})\}\subseteq\Med(\mathcal{X} ;D_{FP})$.
	
	Now let $B\in\mathcal{F}_{c}(\mathbb{R})$ such that $B\neq \med_{Si}(\mathcal{X})$. It implies that there exists $u\in\mathbb{S}^{0}$ and $\alpha\in [0,1]$ such that $s_{B}(u,\alpha)\neq s_{\med_{Si}(\mathcal{X})}(u,\alpha) = \med(s_{\mathcal{X}}(u,\alpha))$. Thus $|s_{B}(u,\alpha) - \med(s_{\mathcal{X}}(u,\alpha))| > 0$ and $O(B;\mathcal{X}) > 0$. It implies that $D_{FP}(B;\mathcal{X}) < D_{FP}(\med_{Si}(\mathcal{X});\mathcal{X}) = 1$ and $\Med(\mathcal{X} ;D_{FP}) = \Med_{s}(\mathcal X) = \{\med_{Si}(\mathcal{X})\} = \{\med_{Gr}(\mathcal{X})\}$.
\end{proof}

The following is a basic fact about fuzzy numbers.

\begin{lemma}\label{lemmadifferentfuzzy}
	Let $A,B\in\mathcal{F}_{c}(\mathbb{R})$ and set
	$$C_{A,B} = \{(u,\alpha)\in\mathbb{S}^{0}\times [0,1] : s_{A}(u,\alpha)\neq s_{B}(u,\alpha)\}.$$
	If $A\neq B$, then $(\upsilon\otimes\nu) (C_{A,B}) > 0$, where $\upsilon$ denotes the uniform distribution measure over $\mathbb{S}^{0}$ and $\nu$ denotes the Lebesgue measure over $[0,1]$.
\end{lemma}

\begin{proof}
	Let $A,B\in\mathcal{F}_{c}(\mathbb{R})$ be two distinct fuzzy sets, thus $C_{A,B}\neq\emptyset$. Reasoning by contradiction, assume $(\upsilon\otimes\nu)(P) = 0$. Let $S = \mathbb{S}^{0}\times [0,1]\setminus C_{A,B}$ be the set of all $(u,\alpha)$ where $s_{A}(u,\alpha) = s_{B}(u,\alpha)$. Since $(\upsilon\otimes\nu)(S) = (\upsilon\otimes\nu)(\mathbb{S}^{0}\times [0,1])$, the set $S$ is dense in $\mathbb{S}^{0}\times [0,1]$. Set $S_{1} = (\{1\}\times [0,1])\cap S$ and $S_{-1} = (\{-1\}\times [0,1])\cap S$. Clearly, $S_{1}$ is dense in $\{1\}\times [0,1]$ and $S_{-1}$ is dense in $\{-1\}\times [0,1]$.
	
	Since $C_{A,B}$ is non-empty, let $(u,\alpha)\in C_{A,B}$. We will prove the case $u=1$, the other one being analogous. We distinguish two cases.
	
	\textbf{Case I. $(\alpha > 0)$.} As $S_{1}$ is dense in $\{1\}\times [0,1]$, there exists a monotone increasing sequence $\{\alpha_{n}\}_{n}$ such that $\lim_{n}\alpha_{n} = \alpha$ and $(1,\alpha_{n})\in S$ for all $n\in\mathbb{N}$. The support function is left-continuous with respect to $\alpha$ (see \cite{MingSupport}), whence
	$$s_{B}(1,\alpha) = \lim_{n\rightarrow\infty} s_{B}(1,\alpha_{n}) = \lim_{n\rightarrow\infty} s_{A}(1,\alpha_{n}) = s_{A}(1,\alpha),$$
	where the second equality is due to $(1,\alpha_{n})\in S$ for every $n\in\mathbb{N}$. Thus $(1,\alpha)\in S$, which is a contradiction.
	
	\textbf{Case II. $(\alpha = 0)$.} In this case, there exists a monotone decreasing sequence $\{\alpha_{n}\}_{n}$ such that $\lim_{n}\alpha_{n} = 0$ and $(1,\alpha_{n})\in S$ for all $n\in\mathbb{N}$. The support function of a fuzzy set is continuous at $\alpha = 0$ (see \cite{MingSupport}), thus
	$$s_{B}(1,0) = \lim_{n\rightarrow\infty} s_{B}(1,\alpha_{n}) = \lim_{n\rightarrow\infty} s_{A}(1,\alpha_{n}) = s_{A}(1,0),$$
	like in Case I, thereby reaching a contradiction.
\end{proof}

The next result is proved in \cite{simplicialsegundo} and it is used in the proof of Theorem \ref{teoremaSimplicial1mediana}.

\begin{lemma}(\cite[Proposition 5.1]{simplicialsegundo})\label{integrandoDFS}
	Let $\mathcal{X}\in L^{0}[\mathcal{F}_{c}(\mathbb{R}^{p})]$, $U\in\mathcal{F}_{c}(\mathbb{R}^{p})$ and $(u,\alpha)\in\mathbb{S}^{p-1}\times [0,1]$. %
	\begin{equation*}
		\begin{aligned}
			\mathbb{P}(s_{U}(u,\alpha)\in [m(u,\alpha), M(u,\alpha&)]) = 1 - [1 - F_{u,\alpha}(s_{U}(u,\alpha))]^{p+1} \\ &- [F_{u,\alpha}(s_{U}(u,\alpha))-\mathbb{P}(s_{\mathcal{X}}(u,\alpha) = s_{U}(u,\alpha))]^{p+1}.
		\end{aligned}
	\end{equation*}
	If, additionally, $\mathcal{X}\in C^{0}[\mathcal{F}_{c}(\mathbb{R}^{p})],$
	\begin{equation*}
		\mathbb{P}(s_{U}(u,\alpha)\in [m(u,\alpha), M(u,\alpha)]) = 1 - [1 - F_{u,\alpha}(s_{U}(u,\alpha))]^{p+1} - [F_{u,\alpha}(s_{U}(u,\alpha))]^{p+1}.
	\end{equation*}
\end{lemma}

\begin{proof}[Proof of Theorem \ref{teoremaSimplicial1mediana}]
	Let $\mathcal{X}\in C^{0}[\mathcal{F}_{c}(\mathbb{R})]$. In \cite{simplicialsegundo} we showed that, for any $U\in\mathcal{F}_{c}(\mathbb{R})$, $D_{mS}$ and $D_{FS}$ can be expressed as
	$$
	D_{mS}(U;\mathcal{X}) = \int_{\mathbb{S}^{0}}\int_{[0,1]}\mathbb{P}[s_{U}(u,\alpha)\in [m(u,\alpha),M(u,\alpha)]] d\nu(\alpha) d\mathcal{V}_{p}(u),
	$$
	and
	$$
	D_{FS}(U;\mathcal{X}) = \inf_{u\in\mathbb{S}^{0}}\int_{[0,1]}\mathbb{P}[s_{U}(u,\alpha)\in [m(u,\alpha),M(u,\alpha)]] d\nu(\alpha).
	$$
	Thus, by Lemma \ref{integrandoDFS} and the fact that $s_{\mathcal{X}}(u,\alpha)$ is a continuous real random variable for each $(u,\alpha)\in\mathbb{S}^{0}\times [0,1]$,
	\begin{equation}\label{DmSecuacion}
		D_{mS}(U;\mathcal{X}) = \int_{\mathbb{S}^{0}}\int_{[0,1]}2F_{u,\alpha}(s_{U}(u,\alpha))[1-F_{u,\alpha}(s_{U}(u,\alpha))]d\nu(\alpha) d\mathcal{V}_{p}(u),
	\end{equation}
	and
	\begin{equation}\label{DFSecaucion}
		D_{FS}(U;\mathcal{X}) = \inf_{u\in\mathbb{S}^{0}}\int_{[0,1]}2F_{u,\alpha}(s_{U}(u,\alpha))[1-F_{u,\alpha}(s_{U}(u,\alpha))]d\nu(\alpha).
	\end{equation}
	
	Now, we prove  $\Med_s(\mathcal X)\subseteq\Med(\mathcal X;D_{mS})$. Let $A$ be a support median. By  \eqref{teoremaBea} and the definition of support median, $\Med_s(\mathcal X)\neq\emptyset$ and there exists at least one support median. Because $s_{\mathcal{X}}(u,\alpha)$ is continuous, $F_{u,\alpha}(s_{A}(u,\alpha)) = 1/2$ for all $(u,\alpha)\in\mathbb{S}^{0}\times [0,1]$. Thus,  $s_{A}(u,\alpha)$ maximizes the integrand  in \eqref{DmSecuacion} for every $(u,\alpha)\in\mathbb{S}^{0}\times [0,1]$. Then $A\in\Med(\mathcal X;D_{mS})$. The proof for $D_{FS}$ is analogous.

	We prove that $\Med(\mathcal X;D_{mS})\subseteq\Med_{s}(\mathcal X)$, since $\Med_1(\mathcal X)=\Med_s(\mathcal X)$ (see Theorem \ref{Teorema1Mediana}), it is equivalent to prove $\Med(\mathcal X;D_{mS})\subseteq\Med_1(\mathcal X)$.
	Assume, reasoning by contradiction, that there exists $V\in\Med(\mathcal X;D_{mS})$ such that $V\not\in\Med_1(\mathcal X)$. Then
	\begin{equation}\label{ecuacion1medianaDFS}
		\text{E}[\rho_{1}(V;\mathcal{X})] > \inf_{U\in\mathcal{F}_{c}(\mathbb{R})}\text{E}[\rho_{1}(U;\mathcal{X})] = \text{E}[\rho_{1}(A;\mathcal{X})]
	\end{equation}
	The set $K = \{(u,\alpha)\in\mathbb{S}^{0}\times [0,1] : s_{V}(u,\alpha)\not\in Med(s_{\mathcal{X}}(u,\alpha))\}$ can be expressed as $K = \{(u,\alpha)\in\mathbb{S}^{0}\times [0,1] : F_{u,\alpha}(s_{V}(u,\alpha))\neq 1/2\}$. In \cite[Proposition 4.6]{jointPedroMiriam} it is proved that the support function of a fuzzy random variable is jointly measurable, thus $K$ is a Lebesgue measurable set. From \eqref{ecuacion1medianaDFS}, $(\upsilon\otimes\nu)(K) > 0$ (see the proof of Theorem \ref{Teorema1Mediana}). The integrand in \eqref{DmSecuacion} is maximized by the medians of $s_{\mathcal{X}}(u,\alpha)$, for each $(u,\alpha)\in\mathbb{S}^{0}\times [0,1]$. As $K$ has positive measure,
	$$
	D_{mS}(V;\mathcal{X}) < D_{mS}(A;\mathcal{X}) = \sup_{U\in\mathcal{F}_{c}(\mathbb{R})} D_{mS}(U;\mathcal{X})
	$$
	and $V\not\in\Med(\mathcal X;D_{mS})$, which is a contradiction.
	
	The proof for $D_{FS}$ is analogous.

\end{proof}

\section{Discussion}\label{discussion}

	We characterized the set of \textit{$1$-medians} as fuzzy sets such that their support functions at $(u,\alpha)$ are univariate medians of the support function of the fuzzy random number at $(u,\alpha)$ for any $(u,\alpha)\in\mathbb{S}^0\times [0,1]$. Thanks to this characterization, the set of \textit{$1$-medians} coincides with what we define as the set of \textit{support medians}, which contains the main pre-existing notions of median for fuzzy random numbers (for instance $\med_{Gr}$ in \cite{GregorMedian} and $\med_{Si}$ in \cite{medianfuzzy1}). Since the $\rho_1$ metric is a natural generalization to the fuzzy setting of the absolute value in $\mathbb{R}$, the notion of \textit{$1$-median} of a fuzzy random number in  \eqref{1medianfuzzynumber} is a natural generalization to fuzzy numbers of the univariate median for real random variables.

	In \cite{medianfuzzy1} it is shown that, for every random fuzzy number, there exists at least one fuzzy number which is a \textit{$1$-median} in the sense of  \eqref{1medianfuzzynumber}, denoted by $\med_{Si}$. The main result of this manuscript is characterize the set of fuzzy numbers which are \textit{$1$-medians} in the sense of \eqref{1medianfuzzynumber}, the set of \textit{support medians}, collected in Theorem \ref{Teorema1Mediana}.
	
	It is possible, in the univiariate case, that real random variables could have infinitely many medians, while in the fuzzy setting there exist different notions of unique median. The set of support medians contains, for every fuzzy random number, the medians $\med_{Gr}$ and $\med_{Si}$, indicating a good generalization of the concept of median for fuzzy random numbers. This set also has, in some cases, infinitely many other elements satisfying the minimization condition presented in \eqref{1medianfuzzynumber}.
	
	The notion of depth arises from the idea of ordering the elements of a space and generalizing the notion of median, as the element that maximizes the depth. There exists some proposals of depth in the fuzzy setting, such as the Tukey depth, $D_{FT}$, the fuzzy simplicial depths, $D_{mS}$ and $D_{FS}$, the projection depth, $D_{FP}$, or the $L^r$-type depths, $D_{r}.$ The second aim of this manuscript is to give a connection between the support medians of a fuzzy random number and the different proposals of depth present in the literature. 
	
	The main result of this manuscript states that the set of support medians and the set of $1$-medians coincide. It implies that the set of elements which maximize $D_1$ depth is the set of support medians (see Proposition \ref{D1support}). Theorem \ref{teoremaTukey1mediana} states that, for every fuzzy random number, the set of elements which maximize the Tukey depth, $D_{FT}$, is the set of support medians and therefore the set of $1$-medians. In the case of the projection depth, we may assume that the fuzzy random number is non-degenerated. In that case, Theorem \ref{teoremaProjection1mediana} assures that the unique element that maximizes $D_{FP}$ is $\med_{Si}$. When we consider the simplicial depths, $D_{mS}$ and $D_{FS}$, we must consider fuzzy random numbers in $C^0[\mathcal{F}_{c}(\mathbb{R}^p)]$. Theorem \ref{teoremaSimplicial1mediana} shows that the sets which maximize $D_{mS}$ and $D_{FS}$ is, exactly, the set of support medians. In the general case this results is not true, as it is shown in Example \ref{ejemplosimpli}.
	
	In this manuscript we give a solution for the minimization problem in \eqref{1medianfuzzynumber} and characterize the set of $1$-medians in terms of the depth proposals in the fuzzy setting. On the one hand, this characterization shows that depth functions in the fuzzy setting are good generalizations since one of the main goals of depth is to generalize the concept of median. On the other hand, $1$-medians of fuzzy random numbers could be computed via depth functions, avoiding the minimization problem in \eqref{1medianfuzzynumber}. 
	
	For future work, it could be interesting to develop some techniques to compute depth proposals, in order to deal with the set of $1$-medians, thanks to the above results which connect the notion of median with the notion of depth for the fuzzy setting. 
	
	As depth functions try to order the elements of a space with respect to a distribution, it is also interesting to study the generalization, to the fuzzy setting, of the notion of quantile via depth functions.

\vskip 1 true cm
{\bf Acknowledgments}
The authors are supported by grant PID2022-139237NB-I00 funded by MCIN/AEI/10.13039/501100011033 and “ERDF A way of making Europe”. Additionally, L. Gonz\'alez was supported by the Spanish Ministerio de Ciencia, Innovaci\'on y Universidades grant MTM2017-86061-C2-2-P. 
P. Ter\'an is also supported by the Ministerio de Ciencia, Innovación y Universidades grant PID2019-104486GB-I00.


\begin{thebibliography}{100}

	
	\bibitem{jointPedroMiriam}
	Alonso de la Fuente, M. \& Ter\'an, P. (2021). Joint measurability of mappings induced by a fuzzy random variable. Fuzzy Sets and Systems 424, 92--104. https://doi.org/10.1016/j.fss.2020.10.007
	


\bibitem{BorDeg} G. Bortolan, R. Degani (1985). A review of some methods for ranking fuzzy subsets. Fuzzy Sets and Systems 15, 1--19.
	
	\bibitem[Chakraborty and Chaudhuri(2014)]{indiosespacial}
	Chakraborty, A. \& Chaudhuri, P. (2014). On data depth in infinite dimensional spaces. Annals of the Institute of Statistical Mathematics 66(2), 303--324. https://doi.org/10.1007/s10463-013-0416-y
	
	\bibitem[Cuesta-Albertos and Nieto-Reyes(2008)]{randomTukey}
	Cuesta-Albertos, J.A. \& Nieto-Reyes, A. (2008). The random Tukey depth. Computational Statistics \& Data Analysis 52(11), 4979--4988. https://doi.org/10.1016/j.csda.2008.04.021
	
	
	\bibitem[DeGroot(1970)]{ResultadoMedianaR}
	DeGroot, M.H. (1970). Optimal statistical decisions. McGraw-Hill, New York.

\bibitem{Den}
H. Deng (2007). A discriminative analysis of approaches to ranking fuzzy numbers in fuzzy decision making. In: J. Lei, J. Yu, S. Zhou (eds.), Fourth International Conference on Fuzzy Systems and Knowledge Discovery (FSKD 2007), volume 1, 22--27. CPS \& IEEE Computer Society, Los Alamitos.

	\bibitem[Diamond and Kloeden(1990)]{diamondkloden}
	Diamond, P. \& Kloeden, P. (1990). Metric spaces of fuzzy sets. Fuzzy Sets and Systems 35(2), 241--249. https://doi.org/10.1016/0165-0114(90)90197-E
	
	\bibitem[Dyckerhoff et al.(1996)]{zonoid}
	Dyckerhoff, R., Koshevoy, G. \& Mosler, K. (1996). Zonoid Data Depth: Theory and Computation. In: Prat, A. (eds) COMPSTAT. Physica-Verlag HD. https://doi.org/10.1007/978-3-642-46992-3\_26
	
	\bibitem[Duque et al.(2015)]{Rafa}
	Duque, R., G\'omez-P\'erez, D., Nieto-Reyes, A. \& Bravo, C. (2015). Analyzing collaboration and interaction in learning environments to form learner groups. Computers in Human Behavior, 47, 42-49. https://doi.org/10.1016/j.chb.2014.07.012
	
	\bibitem[Fraiman and Muñiz(2001)]{medianF2}
	Fraiman, R. \& Muniz, G. (2001). Trimmed means for functional data. Test 10, 419--440. https://doi.org/10.1007/BF02595706
	
	\bibitem{Fre}
	Fr\'echet, M. (1948). Les \'elements al\'eatoires de nature quelconque dans un espace distanci\'e. Annales de l'Institut Henri Poincar\'e 10(4), 215--310.
	
	\bibitem[González-De La Fuente et al.(2022)]{TukeySMPS}
	González-De La Fuente, L., Nieto-Reyes, A. \& Terán, P.(2022). Tukey depth for fuzzy sets. In Building Bridges between Soft and Statistical Methodologies for Data Science (pp. 186--193). Springer, Berlin.
	
	\bibitem[Gonz\'alez-De La Fuente et al.(2022)]{primerarticulo}
	Gonz\'alez-De La Fuente, L., Nieto-Reyes, A. \& Ter\'an, P. (2022). Statistical depth for fuzzy sets.  Fuzzy Sets and Systems 443(A), 58--86. https://doi.org/10.1016/j.fss.2021.09.015
	
	\bibitem[Gonz\'alez-De La Fuente et al.(2022)]{simplicialsegundo}
	Gonz\'alez-De La Fuente, L., Nieto-Reyes, A. \& Ter\'an, P. (2023). Simplicial depths for fuzzy random variables. Fuzzy Sets and Systems 471, 108678. https://doi.org/10.1016/j.fss.2023.108678
	
	
	\bibitem[Gonz\'alez-De La Fuente et al.(2022)]{proyecciontercero}
	Gonz\'alez-De La Fuente, L., Nieto-Reyes, A. \& Ter\'an, P. Projection depth and $L^{r}$-type depths for fuzzy random variables. Submitted for publication.
	
	
	\bibitem[Geenens et al.(2023)]{lensmetric}
	Geenens, G., Nieto-Reyes, A. \& Francisci, G. (2023). Statistical depth in abstract metric spaces. To appear in Statistics and Computing.
	
	\bibitem[Grzegorzewski(1998)]{GregorMedian}
	Grzegorzewski, P. (1998). Statistical inference about the median from vague data. Control and Cybernetics 27(3), 447--464. 
	
	\bibitem[Klir and Yuan(1993)]{Klir}
	Klir, G.J. \& Yuan, B. (1993). Fuzzy sets and fuzzy logic. Theory and applications. Prentice Hall, Upper Saddle River.

	\bibitem[Krätschmer(2001)]{Kra01}
	Krätschmer, V. (2001). A unified approach to fuzzy random variables. Fuzzy sets and systems 123(1), 1--9. https://doi.org/10.1016/S0165-0114(00)00038-5
	
	\bibitem{Kra}
	Kr\"atschmer,V. (2004). Probability theory in fuzzy sample spaces. Metrika 60(2), 167--189. https://doi.org/10.1007/s001840300303
	
	
	\bibitem[Kruse and Meyer(1987)]{krusemeyer}
	Kruse, R. \& Meyer, K.D. (1987). Statistics with Vague Data. Theory and Decision Library B, Springer Netherlands.
	
	\bibitem[Liu(1990)]{LiuSimplicial}
	Liu, R.Y. (1990). On a notion of data depth based on random simplices. Annals of Statistics 18(1), 405--414. https://doi.org/10.1214/AOS/1176347507 
	
	\bibitem[Liu and Modarres(2011)]{medianM3}
	Liu, Z. \& Modarres, R. (2011). Lens data depth and median. Journal of Nonparametric Statistics 23(4), 1063--1074. https://doi.org/10.1080/10485252.2011.584621
	
	\bibitem[L\'opez-Pintado and Romo(2009)]{LopezRomoBand}
	L\'opez-Pintado, S. \& Romo, J. (2009). On the concept of depth for functional data. Journal of the American statistical Association 104(486), 718--734. https://doi.org/10.1198/jasa.2009.0108
	
	\bibitem[Ming(1993)]{MingSupport}
	Ming, M. (1993). On embedding problems of fuzzy number space: part 5. Fuzzy sets and systems, 55(3), 313--318. https://doi.org/10.1016/0165-0114(93)90258-J
	
	\bibitem[Molchanov(2017)]{Mol}
	Molchanov, I. (2017). Theory of random sets, 2nd edition. Springer, London.
	
	\bibitem[Nagy et al.(2016)]{medianF1}
	Nagy, S., Gijbels, I., Omelka, M. \& Hlubinka, D. (2016). Integrated depth for functional data: Statistical properties and consistency. ESAIM - Probability and Statistics 20, 95--130. https://doi.org/10.1051/ps/2016005
	
	\bibitem{NguWu}
	Nguyen, H.T. \& Wu, B. (2006). Fundamentals of Statistics with Fuzzy Data. Springer-Verlag, Berlin.

\bibitem{Ngu}
T. L. Nguyen (2017). Methods in ranking fuzzy numbers: a unified index and comparative reviews. Complexity 2017, article 3083745.
	
	\bibitem[Nieto-Reyes and Battey(2021)]{metricdepth}
	Nieto-Reyes, A. \& Battey, H. (2021). A topologically valid construction of depth for functional data. Journal of Multivariate Analysis 184, 104738. https://doi.org/10.1016/j.jmva.2021.104738
	
	\bibitem[Nieto-Reyes and Battey(2016)]{NietoBattey}
	Nieto-Reyes, A. \& Battey, H. (2016). A topologically valid definition of depth for functional data. Statistical Science 31(1), 61--79. https://doi.org/10.1214/15-STS532
	
	\bibitem[Nieto-Reyes and Cabrera(2020)]{Cabrera}
	Nieto-Reyes, A. \& Cabrera, J. (2020). Statistical Depth based Normalization and Outlier Detection of Gene Expression Data.  \textit{Preprint. arxiv: https://doi.org/10.48550/arXiv.2206.13928}
	
	\bibitem[Puri and Ralescu(1986)]{PuriRalescu}
	Puri, M.L. and Ralescu, D.A. (1986). Fuzzy random variables. Journal of  mathematical analysis and applications 114(2), 409--422. https://doi.org/10.1016/0022-247X(86)90093-4
	
	\bibitem[Shvedov(2016)]{ShvedovMediana}
	Shvedov, A. (2016). Quantile function of a fuzzy random variable and an expression for expectations. Mathematical Notes 100, 477--481. https://doi.org/10.1134/S0001434616090157
	
	\bibitem[Sinova et al.(2012)]{medianfuzzy1}
	Sinova, B., Gil, M.\'A., Colubi, A. \& Van Aelst, S. (2012). The median of a random fuzzy number. The 1-norm distance approach. Fuzzy Sets and Systems 200, 99--115. https://doi.org/10.1016/j.fss.2011.11.004
	
	\bibitem{Bea2}
	Sinova, B., de la Rosa de S\'aa, S. \& Gil, M.\'A. (2013). A generalized L1-type metric between fuzzy numbers for an approach to central tendency of fuzzy data. Information Sciences 242, 22--34. https://doi.org/10.1016/j.ins.2013.03.063
	
	
	\bibitem[Ter\'an(2010)]{PedroDepth}
	Ter\'an, P. (2010). Connections between statistical depth functions and fuzzy sets. In Combining Soft Computing and Statistical Methods in Data Analysis (pp. 611--618). Springer, Berlin, Heidelberg. https://doi.org/10.1007/978-3-642-14746-3\_75
	
	\bibitem[Tukey(1975)]{tukey}
	Tukey, J. (1975). Mathematics and picturing data. In: R.D. James, (ed.) Proceedings of the International Congress of Mathematicians, 2, 523–-531. Canadian Mathematical Congress, Montreal, QC.
	
	\bibitem[Vardy and Zhang(2000)]{medianM2}
	Vardy, Y. \& Zhang, C.H. (2000). The multivariate $L_1$-median and associated data depth. Proceedings of the National Academy of Sciences of the United States of America 97(4), 1423--1426. https://doi.org/10.1073/pnas.97.4.1423
	
	\bibitem[Zadeh(1965)]{zadehfuzzysets}
	Zadeh, L.A. (1965). Fuzzy sets. Information and control 8(3), 338--353. https://doi.org/10.1016/S0019-9958(65)90241-X
	
	\bibitem[Zadeh(1975)]{zadehextension}
	Zadeh, L.A. (1975). The concept of a linguistic variable and its application to approximate reasoning I. Information Sciences 8(3), 199--249. https://doi.org/10.1016/0020-0255(75)90036-5
	
	\bibitem[Zuo (2003)]{medianM1}
	Zuo, Y. (2003). Projection-based depth functions and associated medians. The Annals of Statistics 31(5), 1460--1490. https://doi.org/10.1214/aos/1065705115
	
	\bibitem[Zuo and Serfling(2000)]{ZuoSerfling}
	Zuo, Y. and Serfling, R. (2000). General notions of statistical depth function. Annals of Statistics 28(2), 461--482.  https://doi.org/10.1214/aos/1016218226
	
\end{thebibliography}
\end{document}